\documentclass[12pt,final]{amsart}
\usepackage{amssymb,amsmath}
\usepackage{hyperref}

\theoremstyle{plain}
\newtheorem{theorem}{Theorem}[section]
\newtheorem{proposition}[theorem]{Proposition}
\newtheorem{corollary}[theorem]{Corollary}
\newtheorem{lemma}[theorem]{Lemma}

\theoremstyle{definition}
\newtheorem{definition}[theorem]{Definition}
\newtheorem{example}[theorem]{Example}
\newtheorem{remark}[theorem]{Remark}

\theoremstyle{remark}

\numberwithin{equation}{section}

\newcommand{\N}{\mathbb N}
\newcommand{\Z}{\mathbb Z}

\newcommand{\R}{\mathbb R}
\newcommand{\C}{\mathbb C}
\newcommand{\Hy}{\mathbb H}

\newcommand{\GL}{\operatorname{GL}}

\newcommand{\Ot}{\operatorname{O}}
\newcommand{\SO}{\operatorname{SO}}
\newcommand{\SU}{\operatorname{SU}}
\newcommand{\U}{\operatorname{U}}
\newcommand{\Sp}{\operatorname{Sp}}

\newcommand{\Spin}{\operatorname{Spin}}

\newcommand{\so}{\mathfrak{so}}

\newcommand{\op}{\operatorname}

\newcommand{\Ad}{\operatorname{Ad}}
\newcommand{\spec}{\operatorname{Spec}}
\newcommand{\Hom}{\operatorname{Hom}}

\newcommand{\tr}{\operatorname{tr}}
\newcommand{\diag}{\operatorname{diag}}

\newcommand{\ba}{\backslash}
\newcommand{\mult}{\operatorname{mult}}

\newcommand{\norma}[1]{\|{#1}\|_1}
\newcommand{\Cas}{C}

\newcommand{\String}[1]{\mathcal{S}(#1)} 
\newcommand{\NN}{{|\Phi^+|}}

\newcommand{\PP}{\mathcal{P}}
\newcommand{\TT}{\mathcal{T}}
\newcommand{\dir}{\omega}
\newcommand{\pr}{\op{pr}}

\newcommand{\ee}{\varepsilon}

\title[Strong representation equivalence]
{Strong representation equivalence for compact symmetric spaces of real rank one}

\author{Emilio A. Lauret \and Roberto J. Miatello}

\address{Instituto de Matemática (INMABB), Departamento de Matemática, Universidad Nacional del Sur (UNS)-CONICET, Bahía Blanca, Argentina.}
\email{emilio.lauret@uns.edu.ar}

\address{CIEM--FaMAF (CONICET), Universidad Nacional de C\'ordoba, Medina Allende s/n, Ciudad Universitaria, 5000 C\'ordoba, Argentina.}
\email{miatello@famaf.unc.edu.ar}

\subjclass[2010]{58J50, 58J53, 22E46, 22C05.}
\keywords{Representation equivalent, isospectral, $\tau$-spectrum}
\thanks{This research was supported by grants from CONICET, FONCyT and SeCyT. The first named author was supported by the Alexander von Humboldt Foundation}
\date{June 2021}

\begin{document}

\begin{abstract}
Let $G/K$ be a simply connected compact irreducible symmetric space of real rank one. 
For each $K$-type $\tau$ we compare the notions of $\tau$-representation equivalence with $\tau$-isospectrality. 
We exhibit infinitely many $K$-types $\tau$ so that, for arbitrary discrete subgroups $\Gamma$ and $\Gamma'$ of $G$, 
if the multiplicities of $\lambda$ in the spectra of the Laplace operators acting on sections of the induced $\tau$-vector bundles over $\Gamma\backslash G/K$ and $\Gamma'\backslash G/K$ agree for all but finitely many $\lambda$, then $\Gamma$ and $\Gamma'$ are  $\tau$-representation equivalent in $G$ 
(i.e.\ $\dim \Hom_G(V_\pi, L^2(\Gamma\backslash G))=\dim \Hom_G(V_\pi, L^2(\Gamma'\backslash G))$ for all $\pi\in \widehat G$ satisfying $\Hom_K(V_\tau,V_\pi)\neq0$). 
In particular $\Gamma\backslash G/K$ and $\Gamma'\backslash G/K$ are $\tau$-isospectral 
(i.e.\ the multiplicities agree for all $\lambda$). 

We specially study the case of $p$-form representations, i.e. the irreducible subrepresentations $\tau$ of the representation $\tau_p$ of $K$ on the $p$-exterior power of the complexified cotangent bundle $\bigwedge^p T_\C^*M$. 
We show that for such $\tau$, in most cases $\tau$-isospectrality implies $\tau$-representation equivalence. 
We construct an explicit counter-example for $G/K= \operatorname{SO}(4n)/ \operatorname{SO}(4n-1)\simeq S^{4n-1}$. 
\end{abstract}

\maketitle

\tableofcontents

\section{Introduction} \label{sec:intro}
Let $M=G/K$ be a normal homogeneous space, that is, $G$ is a Lie group, $K$ is a compact subgroup of $G$, and $G/K$ has a  $G$-invariant Riemannian metric induced by an $\Ad(G)$-invariant inner product $\langle\cdot,\cdot\rangle$ on  $\mathfrak g$, the Lie algebra of $G$.
Given a finite dimensional representation $(\tau,W_\tau)$ of $K$  one can form the associated  hermitian $G$-homogeneous vector bundle $E_\tau$ on $M$ and  there is a distinguished self-adjoint, second order, elliptic differential operator $\Delta_\tau$ acting on smooth sections of $E_\tau$,  defined by the  Casimir element $C$ associated to $\langle\cdot,\cdot\rangle$.
When $\tau=1_K$, the trivial representation of $K$, $E_\tau$ is the trivial bundle $M\times \C$, smooth sections of $E_\tau$ are in correspondence with complex-valued smooth functions on $G/K$, and $\Delta_{\tau}$ coincides with the Laplace--Beltrami operator on $M$.

Given a discrete cocompact subgroup $\Gamma$ of $G$, the quotient $\Gamma\ba M$ is a compact good orbifold, with a manifold structure in case $\Gamma$ acts freely on $M$ and 
then $\Gamma \backslash M$ inherits a Riemannian metric on $\Gamma\ba M$ from the one on $M$. 
Furthermore, $E_\tau$ naturally induces a vector bundle $E_{\tau,\Gamma}$ on $\Gamma\ba M$, 
whose sections are identified with $\Gamma$-invariant sections of $E_\tau$. 
Thus, the differential operator $\Delta_{\tau,\Gamma}$ given by the restriction of $\Delta_\tau$ to the space of $\Gamma$-invariant smooth sections of $E_{\tau}$, is a self-adjoint, second order, elliptic operator acting on sections of $E_{\tau,\Gamma}$. 
Since $\Gamma\ba M$ is compact, the spectrum of $\Gamma\ba M$ is discrete, non-negative, and every eigenvalue has finite multiplicity.

The spectrum of $\Delta_{\tau,\Gamma}$, that we call the \emph{$\tau$-spectrum} of $\Gamma\ba M$, can be expressed in Lie theoretical terms. 
More precisely, the multiplicity $\mult_{\Delta_{\tau,\Gamma}}(\lambda)$ of a non-negative real number $\lambda$ in $\spec(\Delta_{\tau,\Gamma})$ is given by 
\begin{equation}\label{eq:taumultip}
\mult_{\Delta_{\tau,\Gamma}}(\lambda) = \sum_{(\pi,V_\pi) \in\widehat G: \; \lambda(C,\pi)=\lambda} n_\Gamma(\pi)\, \dim\Hom_K(W_\tau,V_\pi),
\end{equation}
where $\widehat G$ stands for the unitary dual of $G$, $\lambda(C, \pi)$ is the eigenvalue of $\pi(C)$  on $V_\pi$  ($\pi(C)$ acts as a multiple of the identity map on $V_\pi$ since $C$ is in the center of the universal enveloping algebra), and $n_\Gamma(\pi) =\dim\Hom_G(V_\pi, L^2(\Gamma\ba G))\in \N_0:=\N\cup\{0\}$ (i.e.\ the multiplicity of $\pi$ appears in the right regular representation $L^2(\Gamma\ba G)$ of $G$).
Note that the sum in \eqref{eq:taumultip} is indeed over the set $\widehat G_\tau$ of \emph{$\tau$-spherical representations} of $G$, that is, those $\pi\in\widehat G$ satisfying $\Hom_K(W_\tau,V_\pi)\neq 0$. 

It follows immediately from \eqref{eq:taumultip} that $\Gamma\ba M$ and $\Gamma'\ba M$ are \emph{$\tau$-isospectral} (i.e.\ $\Delta_{\tau,\Gamma}$ and $\Delta_{\tau,\Gamma'}$ have the same spectra), if the discrete cocompact subgroups $\Gamma,\Gamma'$ of $G$ satisfy 
\begin{equation}\label{eq:tau-repequiv}
n_{\Gamma}(\pi)= n_{\Gamma'}(\pi)
\quad \text{for all }\pi\in\widehat G_\tau.
\end{equation} 
Discrete cocompact subgroups $\Gamma$ and $\Gamma'$ of $G$ are called \emph{$\tau$-representation equivalent in $G$} when \eqref{eq:tau-repequiv} holds. 
The converse  question comes up naturally, that is
\begin{quote}
Does $\tau$-isospectrality of $\Gamma\ba M$ and $\Gamma'\ba M$ imply that $\Gamma$ and $\Gamma'$ are $\tau$-representation equivalent in $G$?
\end{quote}

In the sequel, we will say that the \emph{representation-spectral converse} is valid for $(G,K,\langle\cdot,\cdot\rangle,\tau )$ when, for {\it every} $\Gamma,\Gamma'$ discrete cocompact subgroups of $G$ such that $\Gamma \ba M$ and $\Gamma'\ba M$ are $\tau$-isospectral, $\Gamma$ and $\Gamma'$ are $\tau$-representation equivalent in $G$. 
We usually abbreviate $(G,K,\langle\cdot,\cdot\rangle,\tau)$ by $(G,K,\tau)$ when $\langle\cdot,\cdot\rangle$ is clear from the context. 
In case $\tau$-isospectrality a.e.\  (i.e.\ $\mult_{\Delta_{\tau,\Gamma}}(\lambda)=\mult_{\Delta_{\tau,\Gamma}}(\lambda)$ for all but finitely many $\lambda>0$) implies $\tau$-representation equivalence {(and hence $\Gamma\ba G/K$ and $\Gamma'\ba G/K$ are $\tau$-isospectral)} we say that the  representation-spectral converse is valid \emph{strongly}.

The problem of whether `isospectrality for all but finitely many eigenvalues forces isospectrality' has attracted attention for quite some time.
For instance, S.T.~Yau posed the question in 1982 for bounded smooth plane domains (see \cite[Ch.~VII--\S{}IV, page 293, Problem 68]{SchoenYau-book}).

Pesce~\cite{Pesce96} studied the validity of the representation-spectral converse and 
proved it to hold for $(G,K,1_K)$ for many homogeneous spaces $M=G/K$.
For instance, when $M$ is a compact or non-compact Riemannian symmetric space of real rank one, and $M=(\Ot(n)\ltimes \R^n)/\Ot(n)\simeq \R^n$  endowed  with the flat metric.

In \cite{LMR-repequiv} the validity of the representation-spectral converse for the exterior representation $\tau_p$ on spaces of constant curvature is considered.
The representation $\tau_p$ of $K$ (see Definition~\ref{def2:tau_p}) satisfies $E_{\tau_p}\simeq \bigwedge^p (T^*M)$ and $\Delta_{\tau_p}$ coincides with the Hodge--Laplace operator acting on $p$-forms.  
The constant curvature spaces were realized as $S^{n}=\Ot(n+1)/\Ot(n)$ with $n$ odd, $\R^n=(\Ot(n)\ltimes \R)/\Ot(n)$ and $\mathbb{H}^n=\SO(n,1)/\Ot(n)$, thus $K\simeq \Ot(n)$ in all cases. 
In this context, $\tau_p$ is the exterior  representation of $\Ot(n)$ on $\bigwedge^p(\C^n)$, where $\C^n$ denotes the standard representation of $\Ot(n)$.
The following  generalization of Pesce's result was proved in \cite[Thm.~1.5]{LMR-repequiv} for any $0\leq p\leq n$:
if $\Gamma\ba M$ and $\Gamma'\ba M$ are $\tau_q$-isospectral for every $0\leq q\leq p$, then $\Gamma$ and $\Gamma'$ are $\tau_q$-representation equivalent in $G$ for every $0\leq q\leq p$. 
Moreover, \cite[Ex.~4.8--10]{LMR-repequiv} show  counterexamples for the representation-spectral converse for a single $\tau_p$ in the flat case $M=\R^n$.

The case when $G$ is compact shows a more rigid structure. 
When $M=S^{2n-1}$, Gornet and McGowan~\cite[\S{}4]{GornetMcGowan06} proved that the converse holds for $(\Ot(2n),\Ot(2n-1),\tau_p)$ for every $p$. 
Moreover, in \cite[Prop.~3.3]{LMR-repequiv}, the converse for $(\Ot(2n),\Ot(2n-1),\tau)$ is shown for a few more irreducible representations $\tau$ of $\Ot(n)$. 
Up to this point, for fixed $G,K$ the converse was known only for finitely many $\tau\in \widehat K$. 

The aim of this article is to {extend the results in \cite{LMR-repequiv}}, by considering the representation-spectral converse on compact irreducible symmetric spaces $G/K$ of real rank one and arbitrary irreducible representations $\tau$ of $K$. 
We {view} each of these spaces realized as a quotient $G/K$ as follows:
\begin{align*}
S^{2n-1}&=\SO(2n) / \SO(2n-1), &
P^n(\C)&=\SU(n+1) / \op{S}(\U(n)\times\U(1)), \\
S^{2n}&=\SO(2n+1) / \SO(2n), &
P^n(\Hy)&=\Sp(n+1) / \Sp(n)\times\Sp(1),\\
&&
P^2(\mathbb O)&= \op{F}_4/\Spin(9).
\end{align*}
By focusing on those eigenvalues $\lambda$ of $\Delta_{\tau}$ such that $\lambda = \lambda(C,\pi)$ for only one representation $\pi \in \widehat G_\tau$, or else for only a representation  $\pi\in \widehat G_\tau$ and its contragredient $\pi^*$ (see \eqref{eq:taumultip}), that we call \emph{tame eigenvalues,} we  prove the validity of the representation-spectral converse for an infinite set of representations $\tau \in\widehat K$.

A main tool  will be  the strong multiplicity one theorem proved in \cite{LM-strongmultonethm}. 
Namely, we will use \cite[Thm.~1.1]{LM-strongmultonethm} (see Theorem~\ref{thm3:finiteSMOTtau} below) which ensures that if $\Gamma$ and $\Gamma'$ are finite subgroups of $G$ and $n_{\Gamma}(\pi) = n_{\Gamma'}(\pi)$ for all but finitely many $\pi\in\widehat G_{\tau}$, then $\Gamma$ and $\Gamma'$ are $\tau$-representation equivalent in $G$. 
 This result is an extension to the compact case of \cite{BhagwatRajan11} and \cite{Kelmer14}, both mainly focused in non-compact symmetric spaces.

Our main task will be to give sufficient conditions for $\tau\in\widehat K$ so that, for any $\Gamma$, all but finitely many eigenvalues of $\Delta_{\tau,\Gamma}$ are tame. 
This implies, together with Theorem~\ref{thm3:finiteSMOTtau}, that for such $\tau$ the strong representation-spectral converse holds (see Theorem~\ref{thm3:tauiso=>tauequiv}).
As a consequence we will prove:

\begin{theorem}\label{thm1:maintheorem}
Let $G/K$ be a compact irreducible simply connected symmetric space of real rank one. 
Then there exist infinitely many $\tau\in\widehat K$ such that, for $\Gamma,\Gamma'$ arbitrary finite subgroups of $G$, if the multiplicities of $\lambda$ in the spectra of $\Delta_{\tau,\Gamma}$ and $\Delta_{\tau,\Gamma'}$ coincide for all but finitely many $\lambda$, then $\Gamma$ and $\Gamma'$ are $\tau$-representation equivalent in $G$ and consequently, $\Gamma\ba G/K$ and $\Gamma'\ba G/K$ are $\tau$-isospectral. 
Thus, the strong representation-spectral converse holds for $(G,K,\tau)$. 
\end{theorem}

An essential fact is that $\widehat G_\tau$ can be written as a finite union of \emph {strings} of representations  for all $\tau$, when $G/K$ is a compact symmetric space of real rank one. 
In fact, Camporesi~\cite{Camporesi05JFA} proved that there is a dominant $G$-integral weight $\dir$ and a finite set $\PP_\tau$ of dominant $G$-integral weights such that 
\begin{equation}
\widehat G_\tau = \bigcup_{\Lambda \in\PP_\tau} \{\pi_{\Lambda+k\dir}: k\in\N_0\}. 
\end{equation}
When $G/K$ is a sphere or a complex projective space, a parametrization of $\PP_\tau$ is given for every $\tau\in\widehat K$ by the classical branching laws (see Lemmas~\ref{lem4:hatG_tauSO(2n)}, \ref{lem5:hatG_tauSO(2n+1)}, \ref{lem6:hatG_tauSU(n+1)}; see also \cite[\S{}2--3]{Camporesi05Pacific}). On the other hand, to our best knowledge, there is no explicit description of $\PP_\tau$ for every $\tau$  in the cases of $P^n(\Hy)$ and $P^2(\mathbb O)$.
This is a main difficulty in applying Theorem~\ref{thm3:tauiso=>tauequiv}.

For $P^n(\Hy)$, we use a branching law by Tsukamoto~\cite{Tsukamoto81} (see Lemma~\ref{lem7:hatG_tauSp}). 
For $P^2(\mathbb O)$, we use a description by Heckman and van Pruijssen~\cite{HeckmanPruijssen16} (see Lemma~\ref{lem8:hatG_tau}).
In both cases, this holds for infinitely many $K$-types $\tau$.

The proof of Theorem~\ref{thm1:maintheorem} is by analysis of the conditions in Theorem~\ref{thm3:tauiso=>tauequiv}.
It follows from Corollaries~\ref{cor4:SO(2n)SO(2n-1)situations}, \ref{cor5:SO(2n+1)SO(2n)situations} for spheres, and from Theorems~\ref{thm6:SUtwo1-jumps}, \ref{thm7:Spone1-jump}, \ref{thm8:F4-infinity} for the projective spaces. 
Under the conditions in the theorem, $\widehat G_\tau$ is a finite union of strings and all but finitely many Laplace eigenvalues are tame. On the other hand, for many choices of $\tau$ there are infinitely many non-tame eigenvalues to handle (e.g., see Table~\ref{table:strings} for $G=\SO(2n)$). In this situation, additional techniques are needed to solve for the $n_\Gamma(\pi)$'s in terms of the multiplicities of the $\lambda(C,\pi)$'s in equation \eqref{eq2:taumultip} (see for instance  Example~\ref{ex4:SO(2n)unknown-cases} and Table~\ref{table:strings}).  
We do construct counterexamples for the representation-spectral converse in some special situations, using automorphisms. 
For instance, for $(G,K)=(\SO(4n),\SO(4n-1))$, Theorem~\ref{thm4:conversefailsSO(4n)} gives $\tau\in\widehat K$ and $\Gamma, \Gamma'$ discrete subgroups of $G$ that are not $\tau$-representation equivalent in $G$, but $\Gamma\ba S^{4n-1}$ and $\Gamma'\ba S^{4n-1}$ are $\tau$-isospectral. (See also Example 2.1.)

As an important special case, we apply the results to the so called {\it $p$-form representations}, i.e.\ the irreducible constituents of the representation $\tau_p$ of $K$ on the $p$-exterior power of the complexified cotangent bundle $\bigwedge^p T_\C^*M$. 
Then $\Delta_{\tau_p,\Gamma}$ coincides with the Hodge--Laplace operator on $\Gamma$-invariant $p$-forms on $M$.

 For such $\tau$, we obtain in Theorems~\ref{thm4:p-isoSO(2n)}, \ref{thm5:p-isoSO(2n+1)}, \ref{thm6:SUp-iso}, \ref{thm7:Spp-iso} and \ref{thm8:F4p-iso} sufficient conditions, so that $\tau$-isospectrality between $\Gamma\ba M$ and $\Gamma'\ba M$ implies that $\Gamma$ and $\Gamma'$ are $\tau$-representation equivalent. 
In the case of spheres  and  complex projective spaces, the strong representation-spectral converse holds for all $p$-form representations. For quaternionic projective spaces and the Cayley plane we give a proof in some cases. 
For instance, when $G/K$ is the $16$-dimensional space $P^2(\mathbb O)$, we show it holds for every $p\neq 5,7,8,9,11$. 
Its proof is based on the detailed calculations of each $p$-spectrum of $P^2(\mathbb O)$ by Mashimo in \cite{Mashimo97} and \cite{Mashimo06}. 
One difficulty in these cases is that the branching formulas are much more difficult to apply. 
In treating this case we complete Table 30 in \cite{Mashimo06} by adding a few  representations that were missing (see Remark 8.5).

As an isolated consequence of the methods in the paper, we extend in Theorem~\ref{thm4:tau-iso3-dim} the classical spectral uniqueness result among $3$-dimensional spherical space forms by Ikeda. 
Namely we show that if $\Gamma\ba S^{3}$ and $\Gamma'\ba S^3$ are two spherical space forms $\tau$-isospectral for any irreducible representation $\tau$ of $\SO(3)$, then $\Gamma\ba S^{3}$ and $\Gamma'\ba S^3$ are isometric.

The paper is organized as follows. 
Section~\ref{sec:preliminaries} contains preliminaries on the spectra of locally homogeneous manifolds.
In Section~\ref{sec:sufficient-conditions}, Theorem~\ref{thm3:tauiso=>tauequiv} gives sufficient conditions on $G$, $K$, $\tau$ for the validity of the representation-spectral converse. This  
result is then applied to $S^{2n-1}$, $S^{2n}$, $P^n(\C)$, $P^n(\Hy)$, and $P^2(\mathbb{O})$ in Sections~\ref{sec:oddspheres}, \ref{sec:evenspheres}, \ref{sec:SU}, \ref{sec:Sp}, and  \ref{sec:F4}, respectively.

\subsection*{Acknowledgements}
The authors wish to thank Nolan Wallach for a helpful comment concerning Remark~\ref{rem2:Wallach}, and to Roberto Camporesi and Maarten van Pruijssen for drawing their attention to different descriptions of the set of $\tau$-spherical representations in several cases.

\section{Preliminaries} \label{sec:preliminaries}

In this section we review standard facts on homogeneous vector bundles and elliptic differential operators acting on sections of these bundles.

Let $G$ be a Lie group and let $K$ be a compact subgroup of $G$, with Lie algebras $\mathfrak g$ and $\mathfrak k$ respectively.
There is a reductive decomposition $\mathfrak g=\mathfrak k\oplus\mathfrak p$,  with  $[\mathfrak k,\mathfrak p]\subset \mathfrak p$. 
The tangent space of $M:=G/K$ at the point $eK$ is identified with $\mathfrak p$.
Consequently,  $G$-invariant metrics on $M$ are in a bijection with the set of $\Ad(K)$-invariant inner product on $\mathfrak p$. 
We will consider the homogeneous metric on $M$ induced by an $\Ad(G)$-invariant inner product $\langle\cdot,\cdot\rangle$ on $\mathfrak g$, so the resulting Riemannian manifold is a so called \emph{normal homogeneous space}.

Let $(\tau,W_\tau)$ be a finite dimensional unitary representation of $K$. 
There is a naturally associated homogeneous vector bundle $E_\tau:= G\times_\tau W_\tau$ on $M$ as $(G\times W_\tau)/\sim$, where $(gk,w)\sim (g,\tau(k)w)$ for all $g\in G$, $w\in W_\tau$, $k\in K$ (see for instance \cite[Ch.~5]{Wallach-book}).
We denote by $[x,w]$ the class of $(x,w)\in G\times W_\tau$ in $E_\tau$. 
The space of smooth sections of $E_{\tau}$ is isomorphic as a $G$-module to $C^\infty(G;\tau):=\{f:G\to W_\tau\,\text{ smooth}: f(x k)= \tau(k^{-1}) f(x)\;\forall\,x\in G,\, k\in K\}$. 
The action of $G$ is given by the left-regular representation  in  both cases. 
The  identification  is given by $f\mapsto (x\mapsto [x,f(x)])$ for $f\in C^\infty(G;\tau)$.

The Lie algebra $\mathfrak g$ of $G$ acts on $C^\infty(G/K;\tau)$  by
$
(Y\cdot f)(x) =\left.\frac{d}{dt}\right|_{t=0} f(x\exp(tY)).
$
This action induces a representation of the universal enveloping algebra $U(\mathfrak g_\C)$ of the complexified Lie algebra $\mathfrak g_\C$ of $\mathfrak g$.
The \emph{Casimir element} $\Cas\in Z(U(\mathfrak g_\C))$ is given by $\Cas=\sum_i X_i^2\in U(\mathfrak g_\C)$ where $\{X_1,\dots,X_n\}$ is any orthonormal basis of $\mathfrak g$ with respect to $\langle\cdot,\cdot\rangle$. 
It induces a self-adjoint, second order, elliptic differential operator $\Delta_{\tau,\Gamma}$ on $C^\infty(G/K;\tau)$ and therefore on the space of smooth sections of $E_{\tau,\Gamma}$.

Let $\Gamma$ be a discrete cocompact subgroup $\Gamma$ of $G$. 
The space $\Gamma\ba M$ is a compact good orbifold, which has no singular points  if  $\Gamma$ acts freely on $M$.
We consider the bundle $E_{\tau,\Gamma}$ on $\Gamma\ba M$ defined by the relation $[\gamma g,w]\sim [g,w]$ for all $\gamma\in \Gamma$ and $[g,w]\in E_\tau$.
The space of smooth sections of $E_{\tau,\Gamma}$ is identified with the space of $\Gamma$-invariant smooth sections of $E_\tau$.

\begin{remark}
In the next section we will assume that $G$ is compact and semisimple. Then the class of discrete cocompact subgroups of $G$ coincides with the class of finite subgroups of $G$. 
This is a large class since, for instance, any arbitrary finite group embeds in the symmetric group $\mathbb S_n$ for $n$ sufficiently large, and consequently it embeds into $\SO(m)$, $\U(m)$ and $\Sp(m)$ for every $m\geq n$. 

The situation for finite subgroups of $G$ acting freely on $G/K$ changes drastically. 
We refer to \cite{Wolf-book} for a comprehensive study in the case where $G/K$ is a compact rank-one symmetric space. 
\end{remark}

We define $\Delta_{\tau,\Gamma}$ as $\Delta_{\tau}$ restricted to $\Gamma$-invariant smooth sections of $E_\tau$. 
So, $\Delta_{\tau,\Gamma}$ is a formally self-adjoint, second order, elliptic differential operator acting on sections of $E_{\tau,\Gamma}$. 
Its spectrum is non-negative and discrete, since $\Gamma\ba M$ is compact. 
Let $\mult_{\Delta_{\tau,\Gamma}}(\lambda)$ denote the multiplicity of $\lambda$ in $\spec(\Delta_{\tau,\Gamma})$. 
One has that 
\begin{equation}\label{eq2:taumultip}
\mult_{\Delta_{\tau,\Gamma}}(\lambda) := \sum_{(\pi,V_\pi) \in\widehat G: \; \lambda(C,\pi)=\lambda} n_\Gamma(\pi)\, \dim\Hom_K(W_\tau,V_\pi),
\end{equation}
where $\widehat G$  is  the unitary dual of $G$, $\lambda(C, \pi)$ is the eigenvalue of $\pi(C)$ on $V_\pi$ and $n_\Gamma(\pi) \in \N_0$ is the multiplicity of $\pi$ in the right regular representation of $G$ on $L^2(\Gamma\ba G)$, that is
\begin{equation}\label{eq2:decomprightregrep}
L^2(\Gamma\ba G) \simeq \bigoplus_{\pi\in\widehat{G}} \, n_\Gamma(\pi)\,V_\pi
\end{equation}
as $G$-modules.
We denote by $\widehat G_\tau$ the subset of $\tau$-spherical representations of $\widehat G$, that is, $\widehat G_\tau= \{(\pi,V_\pi) \in\widehat G: \Hom_K(W_\tau,V_\pi)\neq0\}$.

We recall some notions from the introduction.
Two spaces $\Gamma\ba M$, and $\Gamma'\ba M$ are said to be \emph{$\tau$-isospectral} if the operators $\Delta_{\tau,\Gamma}$ and $\Delta_{\tau,\Gamma'}$ have the same spectrum and two discrete cocompact subgroups $\Gamma$ and $\Gamma'$ of $G$ are \emph{$\tau$-representation equivalent in $G$} if $n_{\Gamma}(\pi)= n_{\Gamma'}(\pi)$ for all $\pi\in\widehat G_\tau$. 
From \eqref{eq2:taumultip} it follows that, if $\Gamma,\Gamma'$ are $\tau$-representation equivalent in $G$, then $\Gamma\ba M$ and $\Gamma'\ba M$ are $\tau$-isospectral.
We are interested in the validity of the converse assertion. We say that the \emph{representation-spectral converse} is valid for $(G,K,\langle\cdot,\cdot\rangle,\tau )$ if, for {\it every} $\Gamma,\Gamma'$ discrete cocompact subgroups of $G$ such that $\Gamma \ba M$ and $\Gamma'\ba M$ are $\tau$-isospectral, it holds that $\Gamma$ and $\Gamma'$ are $\tau$-representation equivalent in $G$.

We observe that we cannot expect that the representation-spectral converse holds in full generality as the following simple example shows. 
However, we will see that in other cases this question is interesting and far from been fully understood. 
\begin{example}
	For $H$ any compact connected semisimple Lie group, we define groups $G$ and $K$ as follows:
	\begin{equation}
	K:= \diag(H)= \{(h,h):h\in H\}\subset H\times H=:G.
	\end{equation}
	We consider on $G$ any bi-invariant metric induced by an inner product $\langle\cdot,\cdot\rangle$ on $\mathfrak g=\mathfrak h\times \mathfrak h$ satisfying $\langle (X,0),(X,0) \rangle =\langle (0,X),(0,X) \rangle$ for all $X\in\mathfrak h$.
	Let $\varphi_i:H\to G$ be given by $\varphi_1(h)=(h,e)$ and $\varphi_2(h)=(e,h)$.
	
Let $\Gamma$ be any non-trivial finite subgroup of $H$, and write $\Gamma_i=\varphi_i(\Gamma)$ for $i=1,2$.
Clearly, the spaces $\Gamma_1\ba G/K$ and $\Gamma_2\ba G/K$ are isometric and $\tau$-isospectral for all $\tau\in\widehat K$.
We claim that $\Gamma_1$ and $\Gamma_2$ are not $\tau$-representation equivalent in $G$ for any $\tau\in\widehat K  \smallsetminus {1}$, showing that here, the representation-spectral converse for $(G,K,\tau)$ is not valid.

	Let $\tau\in\widehat K=\widehat H$ non-trivial.
	We need to find $\pi\in \widehat G_{\tau}$ satisfying that $n_{\Gamma_1}(\pi) \neq n_{\Gamma_2}(\pi)$.
	Set $\pi=\tau\otimes 1_H$.
	It is clear that $\pi\in \widehat G_\tau$ (since $\pi|_{K}=\tau$) and furthermore,
	\begin{equation}
	\begin{aligned}
	n_{\Gamma_1}(\pi) & = \dim V_{\pi}^{\Gamma_1} =\dim (V_\tau\otimes V_{1_H})^{\Gamma_1} = \dim V_\tau^\Gamma \dim V_{1_H}= \dim V_{\tau}^\Gamma,\\
	n_{\Gamma_2}(\pi) & = \dim V_{\pi}^{\Gamma_2} =\dim (V_\tau\otimes V_{1_H})^{\Gamma_2} = \dim V_\tau\rule{3pt}{0pt} \dim  V_{1_H}^\Gamma= \dim V_\tau.
	\end{aligned}
	\end{equation}
	The assertion follows provided $\dim V_\tau>\dim V_\tau^\Gamma$; such $\Gamma$ always exists since $\tau$ was assumed non-trivial.
	
	We observe that this example does not work for $\tau=1_K$.
	One can check that $\widehat G_{1_K} = \{\sigma\otimes\sigma^*: \sigma\in\widehat H\}$ by using the orthogonality relations of characters (cf.\ \cite[\S1.1 Ex.~2]{Takeuchi}).
	One has that $n_{\Gamma_1}(\sigma\otimes \sigma^*)=\dim V_{\sigma}^\Gamma\dim V_{\sigma^*}$ and $n_{\Gamma_2}(\sigma \otimes\sigma^*)=\dim V_{\sigma}\dim V_{\sigma^*}^\Gamma$, thus $n_{\Gamma_1}(\sigma \otimes \sigma^*)= n_{\Gamma_2}(\sigma \otimes \sigma^*)$ since $\dim V_{\sigma}^\Gamma = \dim V_{\sigma^*}^\Gamma$ by Remark~\ref{rem2:Wallach} below.
\end{example}

\begin{remark}\label{rem2:Wallach}
Since $n_{\Gamma}(\pi)=n_{\Gamma}(\pi^*)$ for every $\pi\in\widehat G$ and every $\Gamma$, two discrete cocompact subgroups of $G$ are $\tau$-representation equivalent in $G$ if and only if they are $\tau^*$-representation equivalent in $G$.
Here, $\pi^*$ and $\tau^*$ denote the contragradient representations of $\pi$ and $\tau$ respectively. 
\end{remark}

\begin{remark}\label{rem2:tau-reducible}
Suppose $(\tau,W_\tau)$ is a finite-dimensional representation of $K$, say $W_\tau\simeq W_{1}\oplus \dots \oplus W_{\ell}$ with each $W_{i}$ $K$-invariant.
One clearly has that $\widehat G_{\tau} = \bigcup_{i=1}^\ell \widehat G_{\tau_i}$, where $\tau_i=\tau|_{W_i}$ for each $i$, 
hence,  $\Gamma$ and $\Gamma'$ are $\tau$-representation equivalent in $G$, if and only if they are  $\tau_i$-representation equivalent in $G$ for every $i$. 
\end{remark}

\begin{definition}\label{def2:problems}
We say that the \emph{strong representation-spectral converse} is valid for the triple $(G,K,\tau)$ if, for every pair $\Gamma,\Gamma'$ of discrete cocompact subgroups of $G$, $\tau$-isospectrality between $\Gamma\ba M$ and $\Gamma'\ba M$ a.e.\ in $\lambda$, implies that $\Gamma$ and $\Gamma'$ are $\tau$-representation equivalent in $G$. 
\end{definition}

We will be specially interested on the so called $p$-form representations.
The next facts are well known (see for instance \cite[\S{}1--2]{IkedaTaniguchi78}).
We fix a normal homogeneous space $G/K$ with $G$ semisimple and $K$ compact, thus $\mathfrak g=\mathfrak k\oplus \mathfrak p$, with $\mathfrak k$ and $\mathfrak p$ orthogonal with respect to $\langle\cdot,\cdot\rangle$.

\begin{definition}\label{def2:tau_p}
Let $\tau_1:K\to \GL(\mathfrak p_\C^*)$ be the contragradient of the complexification of the representation $\op{Ad}:K\to \GL(\mathfrak p)$ and for each $0\leq p\leq \dim(G/K)$, let $\tau_p :K\to \GL(\bigwedge^p(\mathfrak p_\C ^*))$ be the $p$-exterior representation of $\tau_1$. 
\end{definition}

The homogeneous vector bundle $E_{\tau_p}=G \times_{\tau_p} \bigwedge^p(\mathfrak p_\C ^*)$ is identified  with $\bigwedge^p (T^*G/K)_\C$, the $p$-exterior product of the cotangent bundle. 
Moreover, the associated differential operator $\Delta_{\tau_p}$ acting on smooth sections of $E_{\tau_p}$ corresponds to the Hodge--Laplace operator $d d^*+d^* d$ acting on smooth $p$-forms.
Consequently, the notions of $\tau_p$-isospectrality  coincides with the standard notion of $p$-isospectrality.

\section{The representation-spectral converse}
\label{sec:sufficient-conditions}
The goal of this section is to give sufficient conditions for the strong representation-spectral converse to be valid (see Theorem~\ref{thm3:tauiso=>tauequiv}). 
This theorem will be applied to each compact symmetric space of real rank one in the next sections. 
A main tool in the proof is Theorem~\ref{thm3:finiteSMOTtau}, a strong multiplicity one result proved in \cite{LM-strongmultonethm}.

Throughout the section, $G$ denotes a compact connected  semisimple Lie group.
For a maximal torus $T$ of $G$ with Lie algebra $\mathfrak t$, let $\Phi(\mathfrak g_\C,\mathfrak t_\C)$ denote the set of roots with respect to the Cartan subalgebra $\mathfrak t_\C$ of $\mathfrak g_\C$ and let $\Phi^+(\mathfrak g_\C,\mathfrak t_\C)$ be the subset of positive roots relative to an order on $\mathfrak t_\C^*$.
By the highest weight theorem, $\widehat G$ is parametrized by the set $\PP^{+}(G)$ of $G$-integral dominant weights with respect to $\Phi^{+}(\mathfrak g_\C,\mathfrak t_\C)$.

As in \cite{LM-strongmultonethm}, for $\dir,\Lambda_0\in \PP^{+}(G)$, we call the ordered set 
\begin{equation}\label{eq3:string}
\String{\dir, \Lambda_0} := \{\pi_{\Lambda_0+k\dir}:k\in \N_0\}
\end{equation}
the \emph{string of representations} with base $\Lambda_0$ and direction $\dir$.
Its usefulness is clear from the next result.

\begin{lemma} \cite [Lem.~3.3]{LM-strongmultonethm}
Let $G$ be a compact connected semisimple Lie group, let $K$ be a closed subgroup of $G$, and let $\tau$ be a finite dimensional representation of $K$.
Then there {exist} $\omega\in\PP^+(G)$ and a subset $\PP_\tau$ of $\PP^+(G)$ such that 
\begin{equation}\label{eq:unionstrings}
\widehat G_\tau = \bigcup_{\Lambda_0\in \PP_\tau }\, \String{\dir, \Lambda_{0}} = \bigcup_{\Lambda_0\in \PP_\tau } \; \{ \pi_{\Lambda_{0}+k\dir} :k\in \N_0 \}. 	
\end{equation}
Furthermore, this union is disjoint.
\end{lemma}

We sketch briefly the argument in the proof. 
Let $\Lambda$ be such  that $\pi_\Lambda \in \widehat G_\tau$. 
If $\dir$ is the highest weight of any  irreducible non-trivial representation in $\widehat G_{1_K}$, then $\pi_{\Lambda + k\dir}\in \widehat G_\tau$ for any $k\in \N_0$. We may assume that the strings in \eqref{eq:unionstrings}, pairwise, are not contained in each other and, since they have the same direction, this easily implies that the union is disjoint.

Since $G$ is compact, \eqref{eq2:taumultip} tells us that every eigenvalue of $\Delta_{\tau,\Gamma}$ lies in the countable set 
\begin{equation}
\mathcal E_\tau:=\{\lambda(C,\pi): \pi\in\widehat G_\tau\}.
\end{equation}

\begin{definition}\label{def3:tame}
We call an eigenvalue $\lambda \in \mathcal E_\tau$ \emph{tame} if $\lambda = \lambda(C,\pi)$ for just one representation $\pi\in\widehat G_\tau$,  or $\lambda = \lambda(C,\pi) = \lambda(C,\pi^*)$  for just  one pair $(\pi,\pi^*)$ with $\pi\in \widehat G_\tau$ and $\pi\not\simeq \pi^*$.
\end{definition}

Write $[\tau:\pi|_K]=\dim\Hom_K(\tau,\pi)$ for any $\pi$ in $\widehat G$.
We claim that for a tame eigenvalue $\lambda=\lambda(C,\pi)$, the multiplicity $\mult_{\Delta_{\tau, \Gamma}}(\lambda(C,\pi))$, determines $n_{\Gamma}(\pi)$.
In fact, \eqref{eq2:taumultip} gives either 
\begin{align}\label{eq3:mult-simple}
	\mult_{\Delta_{\tau, \Gamma}}(\lambda(C,\pi)) &= [\tau: \pi|_K]\, n_\Gamma (\pi),\text{ or else} 
	\\
	\label{eq4:mult-doble}
	\mult_{\Delta_{\tau, \Gamma}}(\lambda(C,\pi)) &=  ([\tau: \pi|_K]+ [\tau: \pi^*|_K])\, n_\Gamma (\pi).
\end{align}  
Here we use that $n_{\Gamma}(\pi) = n_{\Gamma}(\pi^*)$ (see Remark~\ref{rem2:Wallach}). 

We are now in a position to explain in advance the rough strategy we shall follow. 
We need conditions on $G$, $K$ and $\tau$ to guarantee that all but finitely many eigenvalues in $\mathcal E_\tau$ are tame.
In this situation, if $\Gamma\ba G/K$ and $\Gamma'\ba G/K$ are $\tau$-isospectral, then $n_{\Gamma}(\pi) = n_{\Gamma'}(\pi)$ for all but finitely many $\pi\in\widehat G_\tau$, which yields that $\Gamma$ and $\Gamma'$ are $\tau$-representation equivalent by the strong multiplicity one theorem proved in \cite{LM-strongmultonethm} (Theorem~\ref{thm3:finiteSMOTtau} below).
To do that, we will assume that the decomposition \eqref{eq:unionstrings} of $\widehat G_\tau$ as a disjoint union of strings is finite. 
The rest of the conditions needed will naturally appear from the study of the coincidences among the eigenvalues occurring in the strings. 
This, we will do next.

We still denote by $\langle\cdot,\cdot\rangle$ the Hermitian extension of $\langle\cdot,\cdot\rangle|_{\mathfrak t}$ to $\mathfrak t_\C$ and  $\mathfrak t_\C^*$, of the given inner product $\langle\cdot,\cdot\rangle$ on $\mathfrak g$. 
Let $\Lambda_\pi$ be the highest weight of $\pi$ and let $\rho_G$ denote half the sum of the positive roots relative to $\Phi^+(\mathfrak g_\C,\mathfrak t_\C)$. 
It is well known that the Casimir element $\Cas$ acts by the scalar 
\begin{equation}\label{eq3:lambda(C,pi)}
\lambda(\Cas,\pi):= \langle\Lambda_\pi+\rho_G,\Lambda_\pi+\rho_G\rangle-\langle\rho_G,\rho_G\rangle
= \langle\Lambda_\pi,\Lambda_\pi+2\rho_G\rangle
\end{equation}
on each $\pi\in\widehat G$ (see for instance \cite[Lemma~5.6.4]{Wallach-book}).
Thus, for $\dir,\Lambda_0\in\mathcal P^+(G)$, we have that
\begin{align}\label{eq3:lambda-tira}
\lambda(C,\pi_{k\dir + \Lambda_0}) 
	&= 
	k^2 \langle \dir, \dir \rangle + 2k \langle \dir, \Lambda_0 + \rho_G \rangle + \langle \Lambda_0,\Lambda_0+2\rho_G\rangle.
\end{align}

	It is easy to see that there are no coincidences of Casimir eigenvalues for two representations in the same string, that is, $\lambda(C,\pi_{k\dir + \Lambda_0})=\lambda(C,\pi_{h\dir + \Lambda_0})$ if and only if $k=h$.
	Indeed, by \eqref{eq3:lambda-tira} we have
	\begin{align*}
	0 = \lambda(C,\pi_{k\dir + \Lambda_0})-\lambda(C,\pi_{h\dir + \Lambda_0}) = (k-h)\big( (k+ h)  \langle \dir, \dir \rangle +\langle \dir, 2\rho_G \rangle +2\langle \dir,\Lambda_0 \rangle \big),
	\end{align*}
	hence $k-h=0$ since  the second factor in the right-hand side  is positive.

For $\Lambda_0, \dir \in \mathcal P ^{+}(G)$ we set
\begin{equation}
\mathcal E (\dir, \Lambda_0) =\{\lambda(C, \pi_{k\dir +\Lambda_0}): k \in \N_0\}. 
\end{equation} 
Thus, if \eqref{eq:unionstrings} holds, we have that 
$
\mathcal E_\tau= \bigcup_{\Lambda_0 \in\PP_\tau } \mathcal E (\dir, \Lambda_{0}). 
$

\begin{proposition}\label{prop3:stringeigenvalues}
Let $G$ be a compact connected semisimple Lie group, and let $\dir,\Lambda_0,\Lambda_0'\in \mathcal P ^{+}(G)$ be such that the strings $\String{\dir, \Lambda_0},\String{\dir, \Lambda'_0}$ are disjoint. 
Then $\mathcal E (\dir, \Lambda_0)\cap \mathcal E (\dir, \Lambda'_0)$ is an infinite set if and only if one of the following conditions holds:
\begin{enumerate}
\renewcommand{\labelenumi}{(\roman{enumi})}

\item
$\langle \dir,\Lambda_0\rangle = \langle \dir,\Lambda'_0\rangle$ and $\lambda(C,\pi_{\Lambda_0}) = \lambda(C,\pi_{\Lambda'_0})$.
In this case $\lambda(C,\pi_{k\dir+\Lambda_0 }) = \lambda(C,\pi_{ k \dir+\Lambda_0 '})$ for all $k\in\N_0$ and no other coincidences of eigenvalues of $\String{\dir, \Lambda_0}$ and $\String{\dir, \Lambda'_0}$ occur.

\item $\langle \dir,\Lambda_0\rangle \neq  \langle \dir,\Lambda'_0\rangle$ and
\begin{equation}\label{eq3:condrara} 	
m:= \dfrac{\langle \dir ,\Lambda_0-\Lambda'_0\rangle}{\langle \dir,\dir \rangle}
= \dfrac{\langle \Lambda_0-\Lambda'_0, \Lambda_0+\Lambda'_0 +2\rho_G\rangle} {\langle \dir ,\Lambda_0+\Lambda'_0+2\rho_G\rangle} \in\Z\smallsetminus\{0\}.
\end{equation}
In this case we have that $\lambda(C,\pi_{\Lambda_0 + (k+m) \dir}) = \lambda(C,\pi_{\Lambda'_0 + k\dir})$ if $m>0$ and  $\lambda(C,\pi_{\Lambda_0 + k \dir}) = \lambda(C,\pi_{\Lambda'_0 + (k-m)\dir})$ if $m<0$. No other coincidences of eigenvalues of $\String{\dir, \Lambda_0}$ and $\String{\dir, \Lambda'_0}$ occur.
\end{enumerate}	
\end{proposition}	

\begin{proof}
Set $P(k)=\lambda(C,\pi_{k\dir+\Lambda_0})$ and $Q(k)=\lambda(C,\pi_{k\dir+\Lambda_0'})$. 
Thus, we must show that  equation $P(k) = Q(h)$ has finitely many solutions $(k,h) \in\N_0^2$, except in the cases (i) and (ii) listed in the  proposition. 

In light of \eqref{eq3:lambda-tira}, it is easy to see that the 
exceptions (i), (ii) given in the proposition immediately follow from the following lemma by letting $a=\langle \dir, \dir \rangle$, $b=\langle \dir, \Lambda_0+\rho_G\rangle $, $b'=\langle \dir, \Lambda_0'+\rho_G\rangle $, $c=\langle \Lambda_0, \Lambda_0+2\rho_G\rangle $, $c'=\langle \Lambda_0', \Lambda_0'+2\rho_G\rangle$ and thus $ b-b' = \langle \dir, \Lambda_0-\Lambda_0' \rangle$.
Observe that $c-c'= \langle \Lambda_0-\Lambda_0', \Lambda_0+\Lambda_0'+2\rho_G \rangle $ and \eqref{eq:quadratic} corresponds exactly to the last equality in (ii).
This completes the proof of the proposition.
\end{proof}

\begin{lemma}\label{lem3:coincidences}
For real numbers $a,b,b',c,c'$ with $a\neq0$ and $b,b'>0$, let $P(x)=ax^2+2bx+c$, $Q(x)=ax^2+2b'x+c'$ and $\mathcal{A}=\{(k,h)\in \N_0^2: P(k)=Q(h)\}$. 
If $b=b'$, $\mathcal A$ is infinite if and only if $c= c'$ and in this case $\mathcal A=\{(k,k):k\geq0\}$. 
If $b\neq b'$, $\mathcal A$ is infinite if and only if
\begin{equation}
m:=\frac{b-b'}{a} = \frac{c-c'}{b+b'}  \in \Z\smallsetminus\{0\};
\end{equation}
in this case $\mathcal A= \{(k,k+m):k\geq0\}$ if $m>0$ and $\mathcal A=\{(k-m), k):k\geq0\}$ if $m<0$.
\end{lemma}

\begin{proof}
By straightforward manipulations, one checks that the equation $P(x)=Q(y)$ is equivalent to the equation $a(x+b/a)^2-b^2/a+c= a(y+b'/a)^2 -b'^2/a+c'$, and also to 
\begin{align}\label{eq:quadratic}
a\left(x-y+\frac{b-b'}{a}\right) \left(x+y+\frac{b+b'}{a}\right) &= c'-c +\frac{b^2-b'^2}{a}.
\end{align}
	
The first assertion in the lemma  is clear from the previous equation. 
When $b\neq b'$, it follows immediately that $\mathcal A$ is finite if the right-hand side is non-zero. 
Furthermore, when $b\neq b'$ and the right-hand side vanishes, it is clear that $\mathcal A$ is infinite if and only if $m=(b-b')/a  \in\Z\smallsetminus\{0\}$, in which case $P(k)=Q(k+m)$ for all $k\in\N_0$ if $m>0$ and $P(h-m)=Q(h)$ for all $h\in\N_0$ if $m<0$.
\end{proof}

\begin{remark}
In later sections we will give examples for  $G$ equal to $\SO(2n)$, $\SO(2n+1)$, $\SU(n+1)$ and $\Sp(n+1)$ of strings satisfying condition (i), showing infinitely many pairs of inequivalent representations with coincident eigenvalues (see Examples~\ref{ex4:SO(2n)unknown-cases},  \ref{ex5:unknown-casesSO(2n+1)}, \ref{ex6:SUunknowncases1}, \ref{ex7:Spunknowncases}).
	
Condition (ii) in the proposition is shown to hold in many cases for $G=\SU(n+1)$ and $\op{F}_4$, giving infinitely many pairs of inequivalent representations with coincident eigenvalues (see Examples~\ref{ex6:SUunknowncases2}, \ref{ex6:SUp-isocounterex}, \ref{ex8:F4unknowncases}).
\end{remark}

We next recall the strong multiplicity one theorem to be used in the proof of the main result. 

\begin{theorem}\label{thm3:finiteSMOTtau} \cite[Thm.~1.1]{LM-strongmultonethm}
Let $G$ be a compact connected semisimple Lie group, let $K$ be a closed subgroup of $G$, and let $\tau$ be a finite dimensional representation of $K$.
Then, for any $\Gamma,\Gamma'$ finite subgroups of $G$, if $n_\Gamma(\pi) = n_{\Gamma'}(\pi)$ for a.e.\ $\pi \in \widehat G_{\tau}$, then $\Gamma$ and $\Gamma'$ are $\tau$-representation equivalent.
\end{theorem}

Now, by using Theorem~\ref{thm3:finiteSMOTtau} and Proposition~\ref{prop3:stringeigenvalues} we can prove a result on the validity of the strong representation-spectral converse for $(G,K,\tau)$, under some suitable conditions.

\begin{theorem}\label{thm3:tauiso=>tauequiv}
Let $G$ be a compact connected semisimple Lie group, let $K$ be a closed subgroup of $G$ and let $\tau$ be a finite dimensional representation of $K$.
Suppose there exist a finite subset $\PP_{\tau}$ of $\PP^{+}(G)$ and $\dir \in \PP^{+}(G)$ such that 
\begin{equation}\label{eq3:hatG_tau-finito}
\displaystyle \widehat G_{\tau} = \bigcup_{\Lambda_0\in \PP_\tau} \String{\dir, \Lambda_0} = \bigcup_{\Lambda_0\in \PP_\tau}\{\pi_ {\Lambda_0+k\dir}:k\in \N_0\},
\end{equation}
a disjoint union, and furthermore, for any pair $\Lambda_0,\Lambda_0'$ in $\PP_\tau$ conditions \textup{(i)}, \textup{(ii)} in Proposition~\ref{prop3:stringeigenvalues} do not hold unless
there is a non-negative integer $h$ such that, as $G$-modules, 
\begin{equation}\label{eq3:dual}
V_{\pi_{k\dir+{\Lambda'_0}}} \simeq
V_{\pi_{(k+h)\dir+\Lambda_0}}^*  
\quad\text{or}\quad
 V_{\pi_{(k+h)\dir+ {\Lambda'_0}}} \simeq
V_{\pi_{k\dir+\Lambda_0}}^*
\quad\text{for all $k\in \N_0$.}
\end{equation}
Then, all but finitely many  eigenvalues in $\mathcal E_\tau$ are tame, and the strong representation-spectral converse is valid for $(G,K,\tau)$.
\end{theorem}

\begin{proof}
By \eqref{eq3:hatG_tau-finito}, any eigenvalue of $\Delta_{\tau,\Gamma}$ lies in  $\mathcal E_\tau= \bigcup_{\Lambda_0\in\PP_\tau} \String{\dir,\Lambda_0}$, that is, it is of the form $\lambda(C,\pi_{k\dir+\Lambda_0})$ for some $k\in\N_0$  and $\Lambda_0\in\PP_\tau$. 
We have already shown that there can be no coincidences of Casimir eigenvalues for two representations in the same string $\String{\dir,\Lambda_0}$. 
Thus, if $\lambda\in\mathcal E_\tau$ is not tame, then there are at least two representations $\pi$ and $\pi'$ in $\widehat G_{\tau}$ with $\pi\not\simeq \pi'$ and $\pi^*\not\simeq \pi'$ contributing to the multiplicity of the eigenvalue $\lambda =\lambda(C,\pi)=\lambda(C,\pi')$ in \eqref{eq2:taumultip}, with $\pi$ and $\pi'$  belonging to different strings, i.e.\ $\pi \in \String{\dir,\Lambda_0}$ and $\pi' \in \String{\dir,\Lambda_0'}$ for some $\Lambda_0\neq \Lambda_0'$ in $\PP_{\tau}$.

Since $\PP_{\tau}$ is assumed to be finite, in order to prove that all but finitely many eigenvalues in $\mathcal E_\tau$ are tame, we are left with the task of showing that there are only finitely many coincidences among the numbers in $\mathcal E(\dir,\Lambda_0)$ and $\mathcal E(\dir,\Lambda_0')$ for every $\Lambda_0\neq \Lambda_0'$ in $ \PP_\tau$, except in the case when \eqref{eq3:dual} holds, by \eqref{eq4:mult-doble}. 
Proposition~\ref{prop3:stringeigenvalues} tells us that $\mathcal E(\dir,\Lambda_0)\cap \mathcal E(\dir,\Lambda_0')$ cannot have  infinitely many elements since by the hypotheses the conditions (i)-(ii) are not satisfied. 
This proves that all but finitely many eigenvalues in $\mathcal E_\tau$ are tame.

We now show the validity of the strong representation-spectral converse.
Let $\Gamma$ and $\Gamma'$ be finite subgroups of $G$ such that $\mult_{\Delta_{\tau,\Gamma}}(\lambda)= \mult_{\Delta_{\tau,\Gamma'}}(\lambda)$ for all but, possibly, finitely many eigenvalues $\lambda \in \mathcal E_\tau$. 
Thus, this coincidence holds a.e.\ for the set of tame eigenvalues, which has a finite complement. 
Hence, by \eqref{eq4:mult-doble}, $n_\Gamma (\pi)=n_{\Gamma'} (\pi)$ for a.e.\ $\pi \in \widehat G_\tau$ and then $n_\Gamma (\pi)=n_{\Gamma'} (\pi)$ for every $\pi\in\widehat G_\tau$  by Theorem~\ref{thm3:finiteSMOTtau}. 
Consequently, $\Gamma$ and $\Gamma'$ are $\tau$-representation equivalent in $G$.
\end{proof}

One can actually give a refinement of the previous theorem showing that, for any finite subgroup $\Gamma$ of $G$, an adequate finite part of the spectrum of $\Delta_{\tau,\Gamma}$ determines the whole spectrum. 
In particular, given two finite subgroups $\Gamma$ and $\Gamma'$ of $G$,  coincidence of finitely many multiplicities of eigenvalues  of $\Delta_{\tau,\Gamma}$ and $\Delta_{\tau,\Gamma'}$ implies representation equivalence and hence, the validity of the strong representation-spectral converse. 
The details are {given} in the next remark and are based in the application of \cite[Thm.~1.2]{LM-strongmultonethm}, which is a refinement of Theorem~\ref{thm3:finiteSMOTtau}.

\begin{remark} 
Assume $G$, $K$ and $\tau$ are as in Theorem 3.7. 
Let $q$ be a positive integer and let $\mathcal F_{\tau}$ be a finite subset of tame elements in $\mathcal E_{\tau}$ such that 
\begin{equation}\label{eq3:finiteeigenvalues}
|\{k\in \N_0: \lambda(C, \pi_{\Lambda_0+k\dir}) \in \mathcal F_{\tau} \} \cap (j+q\Z)| \geq \NN+1 
\quad\text{for all $j\in\Z$ and $\Lambda_0\in\mathcal P_\tau$}.
\end{equation}
We claim that, for any finite subgroup $\Gamma$ of $G$ such that $|\Gamma|$ divides $q$, the finite part of the spectrum of $\Delta_{\tau,\Gamma}$ associated to $\mathcal F_{\tau}$ (i.e.\ the set of $(\lambda, \mult_{\Delta_{\tau, \Gamma}}(\lambda))$ with $\lambda\in\mathcal F_{\tau}$) determines the multiplicities $n_{\Gamma}(\pi)$ for all $\pi\in\widehat G_\tau$ and hence the whole spectrum of $\Delta_{\tau,\Gamma}$ is determined.

Set $\widehat F_\tau := \{\pi \in\widehat G_\tau:\lambda(C,\pi) \in \mathcal F_{\tau}\}.$ 
Let $\Gamma$ be a finite subgroup of $G$ such that $|\Gamma|$ divides $q$.
For any $\lambda=\lambda(C,\pi) \in\mathcal F_{\tau}$, by \eqref{eq3:mult-simple} and \eqref{eq4:mult-doble}, $n_{\Gamma}(\pi)$ is determined by $\mult_{\Delta_{\tau, \Gamma}}(\lambda)$. 
Consequently, the finite set $\{(\lambda, \mult_{\Delta_{\tau, \Gamma}}(\lambda)) : \lambda\in\mathcal F_{\tau}  \}$ determines the $n_{\Gamma}(\pi)$ for all $\pi\in\widehat F_\tau$. 
Now, \eqref{eq3:hatG_tau-finito} and \eqref{eq3:finiteeigenvalues} ensure that $\widehat F_\tau$ satisfies the assumptions in \cite[Thm.~1.2]{LM-strongmultonethm}.
Hence, the $n_{\Gamma}(\pi)$ for all $\pi\in\widehat G_\tau$ are determined, and therefore also all of the spectrum of $\Delta_{\tau,\Gamma}$. 
Thus, given $\Gamma$, $\Gamma'$ with $|\Gamma|$, $|\Gamma'|$ dividing $q$, the coincidence of  $(\lambda, \mult_{\Delta_{\tau, \Gamma}}(\lambda))$ for $\lambda\in\mathcal F_{\tau}$ implies their coincidence for all $\lambda$, hence the strong representation converse is valid for $(G,K,\tau)$.
\end{remark}

\begin{remark}
	We now observe a problem that is similar  to the representation-spectral converse where Theorem~\ref{thm3:finiteSMOTtau} could be applied. 
	The well-known Sunada method \cite{Sunada85} and its generalization by DeTurck and Gordon~\cite{DeTurckGordon89} produce \emph{strongly isospectral} manifolds, that is, manifolds isospectral with respect to any strongly elliptic natural differential operator acting on sections of a natural bundle. 
	Consequently, if $\Gamma$ and $\Gamma'$ are representation equivalent subgroups of $G$, then $\Gamma\ba M$ and $\Gamma'\ba M$ are $\tau$-isospectral for every $\tau\in\widehat K$. 
	
	The converse was proved in \cite{Pesce95} for $S^n$ and $H^n$ and in \cite{Lauret-strongequivflat} for $\R^n$. 
	It would be of interest to know whether the converse also holds for $P^n(\C)$, $P^n(\Hy)$, and $P^2(\mathbb O)$ by using similar tools as in this article.
	More precisely, for $(G,K)$ equals to $(\SU(n+1),\textrm{S}(\U(n)\times \U(1)))$ or $(\Sp(n+1),\Sp(n)\times\Sp(1))$, and $\Gamma,\Gamma'$ finite subgroups of $G$ satisfying that $\Gamma\ba G/K$ and $\Gamma'\ba G/K$ are $\tau$-isospectral for all $\tau\in\widehat K$, whether $\Gamma$ and $\Gamma'$ are necessarily representation equivalent in $G$, that is, $n_{\Gamma}(\pi)=n_{\Gamma'}(\pi)$ for all $\pi\in \widehat G$.
\end{remark}

\section{Odd dimensional spheres}
\label{sec:oddspheres}
In the rest of this paper, we will give applications of Theorem~\ref{thm3:tauiso=>tauequiv} for each irreducible compact symmetric space $G/K$ of real rank one.
In our study we will find increasing difficulties, mainly due to the fact that the branching formulas in some cases become more intricate and harder to apply.  

Relative to  the $p$-form spectrum, it often involves the contribution of many different irreducible representations, that should be treated separately and do not always behave in the same way.

We begin our case-by-case study by considering odd-dimensional spheres.
Throughout this section, for any $n\geq2$, let $G=\SO(2n)$ and let $K$ be the subgroup $\SO(2n-1)$ embedded in the upper left-hand block in $G$. 
We have that $G/K$ is diffeomorphic to $S^{2n-1}$. 
We consider the metric induced by the inner product $\langle X,Y\rangle = -\tfrac{1}{2} \tr(XY)$ on $\mathfrak g$, which is a negative multiple of the Killing form. 
This metric has constant curvature one.

We pick the maximal torus of $G$ given by
$$
T=\{\diag (R(\theta_1),\dots,R(\theta_n)): \theta_j\in\R \; \forall\, j\}, 
$$
where $R(\theta)=\left(\begin{smallmatrix}\cos\theta & \sin\theta\\ -\sin\theta & \cos\theta \end{smallmatrix}\right)$. 
Thus, every element in $\mathfrak t_\C$ has the form 
$$
X= \diag\left( \left(\begin{smallmatrix}0& i\theta_1 \\ -i\theta_1& 0\end{smallmatrix} \right),\dots,\left( \begin{smallmatrix}0& i\theta_n \\ -i\theta_n& 0\end{smallmatrix} \right) \right) 
$$
with $\theta_j\in\C$ for all $j$. 
Let $\varepsilon_j\in \mathfrak t_\C^*$ given by $\varepsilon_j(X)=\theta_j$ for $X$ as above. 
One has that $\Phi(\mathfrak g_\C,\mathfrak t_\C)= \{\pm \varepsilon_i\pm\varepsilon_j: 1\leq i<j\leq n\}$, $\PP(G)=\bigoplus_{j=1}^n\Z\varepsilon_j$ and $\langle \varepsilon_i,\varepsilon_j\rangle =\delta_{i,j}$.
Furthermore, if we take the lexicographic order on $\mathfrak h^*$ with respect to the basis $\{\varepsilon_j : 1\le j\le n\}$, the simple roots are $\{\varepsilon_1-\varepsilon_{2},\dots,\varepsilon_{n-1}- \varepsilon_{n}, \varepsilon_{n-1}+\varepsilon_n\}$, $\Phi^+(\mathfrak g_\C,\mathfrak t_\C)= \{\varepsilon_i\pm\varepsilon_j: 1\leq i<j\leq n\}$, and furthermore $\PP^{+}(G) = \{\sum_{j=1}^n a_j\varepsilon_j\in \PP(G): a_1\geq\dots\geq a_{n-1}\geq |a_n|\}$.

We pick in $K$ the maximal torus $T\cap K$, thus the associated Cartan subalgebra is $\mathfrak t_\C\cap \mathfrak k_\C$ and $\{\varepsilon_1,\dots,\varepsilon_{n-1}\}$ is a basis of $(\mathfrak t_\C\cap \mathfrak k_\C)^*$. 
Furthermore, the order chosen above gives $\Phi^{+}(\mathfrak k_\C,\mathfrak t_\C\cap\mathfrak k_\C)=\{\varepsilon_i\pm\varepsilon_{j}:1\leq i<j\leq n-1\}\cup \{\varepsilon_i:1\leq i\leq n-1\}$, simple roots $\{\varepsilon_1-\varepsilon_2, \dots, \varepsilon_{n-1}-\varepsilon_{n},\varepsilon_n\}$, and $\PP^{+}(K) = \{\sum_{j=1}^{n-1} b_j\varepsilon_j\in \bigoplus_{j=1}^{n-1}\Z\varepsilon_j: b_1\geq\dots\geq b_{n-1}\geq0\}$. 

We  recall the well-known branching law in the present case (see for instance \cite[Thm.~8.1.4]{GoodmanWallach-bookSpringer} and \cite[Thm.~9.16]{Knapp-book-beyond}). 
If $\tau_\mu \in\widehat K$ and $\pi_\Lambda \in\widehat G$  have highest weights $\mu=\sum_{j=1}^{n-1} b_j\varepsilon_j\in\PP^{+}(K)$, $\Lambda= \sum_{j=1}^{n} a_j \varepsilon_j \in\PP^{+}(G)$ respectively, then $\tau_\mu$ occurs in the decomposition of $\pi|_K$ if and only if
\begin{equation}\label{eq4:SO(2n)entrelazamiento}
a_1\geq b_1\geq a_2\geq b_2\geq \dots \geq a_{n-1}\geq b_{n-1}\geq |a_n|\geq 0.
\end{equation}
Furthermore, when this is the case, then $\dim \Hom_K(\tau_\mu,\pi_{\Lambda})=1$.
As an immediate consequence we obtain the following description of $\widehat G_\tau$ for any $\tau\in\widehat K$.

\begin{lemma}\label{lem4:hatG_tauSO(2n)}
For any $\mu=\sum_{j=1}^{n-1} b_j\varepsilon_j\in \PP^{+}(K)$, we have that 
\begin{equation}\label{eq4:hatSO(2n)_tau}
\widehat G_{\tau_\mu} 
= \bigcup_{\Lambda_0\in \PP_{\tau_\mu}} \String{\varepsilon_1,\Lambda_0}
= \bigcup_{\Lambda_0\in \PP_{\tau_\mu}} \{\pi_ {\Lambda_0+k\varepsilon_1}:k\in \N_0\},
\end{equation}
where
\begin{equation}\label{eq4:P_tauSO(2n)}
\PP_{\tau_\mu} = \left\{ \sum_{j=1}^{n} a_j\varepsilon_j :  a_j\in\Z\;\forall j,\, a_1=b_1,\; \textrm{(\ref{eq4:SO(2n)entrelazamiento}) holds} \right\}.
\end{equation}
Consequently, $\widehat G_{\tau_\mu}$ is a finite disjoint union of strings with  direction $\dir = \varepsilon_1$.
\end{lemma}

Next, we will apply Theorem~\ref{thm3:tauiso=>tauequiv} by using the above choices of $\PP_{\tau_\mu}$ and $\dir$.
We first note that condition (ii)  in Proposition~\ref{prop3:stringeigenvalues} is never a problem in this case since $\langle \dir,\Lambda_0 \rangle = \langle \dir,\Lambda_0' \rangle$ for all $\Lambda_0,\Lambda_0'\in\mathcal P_{\tau_{\mu}}$.
It remains to find conditions on $\mu$ so that $\lambda(C,\pi_{\Lambda_0}) \neq \lambda(C,\pi_{\Lambda_0'}) $ for all $\Lambda_0\neq \Lambda_0'$ in $\mathcal P_{\tau_{\mu}}$. 

Let $\Lambda_0=\sum_{i=1}^{n} a_i \varepsilon_i\in \PP_{\tau_\mu}$.
By \eqref{eq3:lambda(C,pi)}, since $\rho_G=\sum_{i=1}^{n}(n-i)\varepsilon_i$, we have that
\begin{align}\label{eq4:lambda(C,pi_Lambda_0)}
	\lambda(C,\pi_{\Lambda_0}) 
	&= \langle \Lambda_0,\Lambda_0+2\rho\rangle  = b_1(b_1+2(n-1)) + \sum_{i=2}^n a_i(a_i+2(n-i)). 
\end{align}
Therefore, $\lambda(C,\pi_{\Lambda_0}) \neq \lambda(C,\pi_{\Lambda_0'}) $ if and only if $\sum_{i=2}^n a_i(a_i+2(n-i))\neq \sum_{i=2}^n a_i'(a_i'+2(n-i))$.

We now check whether \eqref{eq3:dual} can hold.
It {cannot happen when} $n$ is even since every irreducible representation of $\SO(2n)$ is self-conjugate 
(see for instance \cite[VI.(5.5)(ix)]{BrockerDieck}).
Assume that $n$ is odd. 
For
$\Lambda=\sum_{i=1}^{n} a_i\varepsilon_i \in \PP^{+}(G)$, let $\overline{\Lambda} =\sum_{i=1}^{n-1} a_i\varepsilon_i -a_n\varepsilon_n$, which also lies in $\PP^{+}(G)$. 
Then, $\pi_{\Lambda}^*\simeq \pi_{\overline{\Lambda}}$ for all $\Lambda \in \PP^{+}(G)$ (see for instance \cite[VI.(5.5)(x)]{BrockerDieck}). 
Hence, if $n$ is odd,  \eqref{eq3:dual} holds for $\Lambda_0$, $\Lambda_0'$ if and only if $\Lambda_0'=\overline{\Lambda_0}$ and $a_n\neq 0$.

We thus obtain the following result, as a consequence of Theorem~\ref{thm3:tauiso=>tauequiv}. 

\begin{theorem}\label{thm4:formSO(2n)SO(2n-1)}
	Let $G=\SO(2n)$, $K=\SO(2n-1)$, and $\tau_\mu\in\widehat K$ with $\mu= \sum_{i=1}^{n-1} b_i\varepsilon_i\in \PP^{+}(K)$. 
	Assume the form 
	\begin{equation}\label{eq4:formSO(2n)SO(2n-1)}
	(a_2,\dots,a_n) \longmapsto \sum_{i=2}^{n} a_i(a_i+2(n-i)) 
	\end{equation}
	represents different numbers on the set 
	\begin{equation}\label{eq4:T_GmuSO(2n)}
	\TT_{\SO(2n),\mu} :=
	\begin{cases}
	\{(a_2,\dots,a_n) \in\Z^{n-1}: b_1\geq a_2\geq b_2\geq \dots\geq b_{n-1}\geq |a_n|\}
	& \text{if $n$ is even,}\\ 
	\{(a_2,\dots,a_n) \in\Z^{n-1}: b_1\geq a_2\geq b_2\geq \dots\geq b_{n-1}\geq a_n\geq0\}
	& \text{if $n$ is odd.}\\ 
	\end{cases}
	\end{equation} 
Then, the strong representation-spectral converse is valid for $(\SO(2n),\SO(2n-1),\tau_\mu)$.
\end{theorem}

\begin{remark}\label{rem4:b_n-1=0}
We note that if the assumption in Theorem~\ref{thm4:formSO(2n)SO(2n-1)} holds and $n$ is even, then necessarily $b_{n-1}=0$. 
Indeed, if $b_{n-1}>0$, then the form \eqref{eq4:formSO(2n)SO(2n-1)} represents the same number in the elements $(b_2,\dots,b_{n-1},1)$ and $(b_2,\dots,b_{n-1},-1)$, which lie in $\mathcal T_{G,\mu}$,  and consequently the hypotheses are not satisfied. 
\end{remark}

In \cite[Prop.~3.3]{LMR-repequiv}, the representation-spectral converse was proved for $(\Ot(2n),\Ot(2n-1),\tau)$ for every $\tau\in\widehat {\Ot(2n-1)}$ such that  $\tau|_{\SO(2n-1)}$ has highest weight $\mu= \sum_{i=1}^{n-1} b_i\varepsilon_i\in \PP^{+}(K)$ satisfying that $b_1\leq 2$.
The same proof works for $(\SO(2n),\SO(2n-1),\tau_\mu)$ for every $\mu= \sum_{i=1}^{n-1} b_i\varepsilon_i\in \PP^{+}(K)$ such that $b_1\leq 2$ and $b_{n-1}=0$.
In particular, this gives a finite set of $\tau\in\widehat K$ for which the converse was known.

Now, Theorem~\ref{thm4:formSO(2n)SO(2n-1)} provides infinitely many $\tau_\mu \in\widehat K$ {for which} the representation-spectral converse is valid for $(\SO(2n),\SO(2n-1),\tau_\mu)$, enlarging  the finite set of cases already known.
The next corollary exemplifies this fact by giving  explicit choices of the highest weights $\mu$.
Roughly speaking, in the decreasing sequence of coefficients of the highest weight $\mu = \sum_{i=1}^{n-1} b_i\varepsilon_i$, we may allow three jumps of length one, one jump of arbitrary length, or else, an arbitrary first jump followed by two jumps of length one.

\begin{corollary}\label{cor4:SO(2n)SO(2n-1)situations}
	Let $G=\SO(2n)$, $K=\SO(2n-1)$, and $\tau_\mu\in\widehat K$ with $\mu= \sum_{i=1}^{n-1} b_i\varepsilon_i\in \PP^{+}(K)$.
	Assume $b_{n-1}=0$ if $n$ is even, and  in addition, {any} one of the following conditions is satisfied 
\begin{enumerate} \renewcommand{\labelenumi}{(\roman{enumi})}

\item\label{item4:b_1<=3} $b_1\leq 3$;
		
\item\label{item4:onejump} 
there is $2\leq j\leq n-1$ such that $b_i-b_{i+1}=0$ for all $1\leq i\leq n-1$, $i\neq j$, and $b_n=0$; 
		
\item\label{item4:b_2leq2} $b_2\leq 2$ and $b_1$ arbitrary. 
\end{enumerate}
Then, the strong representation-spectral converse is valid for $(\SO(2n),\SO(2n-1),\tau_\mu)$.
\end{corollary}

\begin{proof}
	The entire proof is straightforward, based on showing that the assumption in Theorem~\ref{thm4:formSO(2n)SO(2n-1)} holds, i.e.\ that the form \eqref{eq4:formSO(2n)SO(2n-1)} represents different numbers on the corresponding set $\TT_{G,\mu}$. 
	We only give some details for case (iii). 
	The other cases are very similar and left to the reader.
	
	Assume $b_2=2$.
	Thus, there are indices $1\leq j\leq k\leq n-1$ such that \begin{equation*}
		b_1\geq b_2 =\dots=b_{j-1}=2 >b_{j} = \dots= b_{k-1}=1>b_{k}=\dots = b_{n-1}=0. 
	\end{equation*}
	We give the details when $j<k$. 
	Let $(a_2,\dots,a_n)$ and $(a_2',\dots,a_n')$ in $\TT_{G,\mu}$, thus $a_i=a_i'=2$ for all $3\leq i\leq j-1$, $a_i=a_i'=1$ for all $j+1\leq i\leq k-1$, $a_i=a_i'=0$ for all $k+1\leq i\leq n$, and furthermore, $b_1\geq a_2,a_2'\geq 2\geq a_j,a_j'\geq 1\geq a_k,a_k'\geq 0$. 
	Hence, if $(a_2,\dots,a_n)$ and $(a_2',\dots,a_n')$ represent the same number under \eqref{eq4:formSO(2n)SO(2n-1)}, we obtain that 
	\begin{align*}
		0 = 
		\sum_{i=2}^{n} \big( a_i(a_i+2(n-i)) -a_i'(a_i'+2(n-i)) \big)
		=\sum_{i\in\{2,j,k\}} (a_i-a_i') \big( a_i+a_i'+2(n-i)\big).
	\end{align*}
	It remains to show that $a_i=a_i'$ for any $i\in \{2,j,k\}$. 
	One can easily check that if $a_i=a_i'$ for some $i\in \{2,j,k\}$, then $a_i=a_i'$ for all $i$. 
	Thus, we assume $a_i\neq a_i'$ for all $i\in \{2,j,k\}$.
	Suppose $a_2>a_2'$, thus $a_{j}-a_{j}'=a_{k}-a_{k}'=-1$, that is, $a_j'=2$, $a_j=1$, $a_k'=1$ and $a_k=0$. 
	Hence 
	\begin{equation*}
		(a_2-a_2')(a_2+a_2'+2(n-2)) = 3+2(n-j) + 1+2(n-k).
	\end{equation*}
	Since the right-hand side is even, it follows that $a_2\pm a_2'$ is also even, in particular $a_2-a_2'\geq 2$. 
	This implies that the left-hand side is strictly bigger than the right-hand side, a contradiction.
\end{proof}

We next show {sets of} disjoint strings $\String{\dir,\Lambda_0}$ and $\String{\dir,\Lambda_0'}$ having the same set of eigenvalues, together with $\tau_\mu\in\widehat K$, such that $\Lambda_0,\Lambda_0'\in\mathcal P_{\tau_\mu}$. Furthermore,   $\mu$ has $b_1=4$, which tells us that the upper bound in condition (i) of Corollary~\ref{cor4:SO(2n)SO(2n-1)situations} is optimal.
Hence, in these cases, the corresponding eigenvalues are not tame and we cannot obtain the strong representation-spectral converse as a consequence of Theorem~\ref{thm3:tauiso=>tauequiv}.

\begin{example}\label{ex4:SO(2n)unknown-cases}
We first consider $n=3$. 
Let $\Lambda_0= 4\varepsilon_1+4\varepsilon_2$ and $\Lambda_0'= 4\varepsilon_1+3\varepsilon_2+3\varepsilon_3$.
{We have that} $\lambda(C,\pi_{k\varepsilon_1+\Lambda_0}) = \lambda(C,\pi_{k\varepsilon_1+\Lambda_0'})$ for all $k\in\N_0$ since condition (i) in Proposition~\ref{prop3:stringeigenvalues} holds, {given that}  $\langle \dir,\Lambda_0\rangle = \langle \dir,\Lambda_0'\rangle= 4$ and $\dir=\varepsilon_1$. 
Furthermore, one can easily check that $\lambda(C,\pi_{\Lambda_0}) = \lambda(C,\pi_{\Lambda_0'}) = 56$ by \eqref{eq4:lambda(C,pi_Lambda_0)}, which proves that $\mathcal E(\dir,\Lambda_0)= \mathcal E(\dir,\Lambda_0')$. 

This coincidence between the Casimir eigenvalues of the strings $\String{\dir,\Lambda_0}$ and $\String{\dir,\Lambda_0'}$ {produces} infinitely many non-tame eigenvalues of $\Delta_{\tau,\Gamma}$ for some finite subgroup $\Gamma$ of $G=\SO(6)$ only if ${\pi_{\Lambda_0}}$ and $ {\pi_{\Lambda_0'}}$ are simultaneously in $\widehat G_\tau$. 
This is the case for $\tau_\mu$ with $\mu= b_1\varepsilon_1+3\varepsilon_2$ for any $b_1\geq4$. 
This obstruction to apply Theorem~\ref{thm3:tauiso=>tauequiv} can be read off from Theorem~\ref{thm4:formSO(2n)SO(2n-1)} since the form \eqref{eq4:formSO(2n)SO(2n-1)} represents the same number at $(4,0,0)$ and at $(3,3,0)$ because $\rho_G=2\varepsilon_1+\varepsilon_2$. 
	
Further, for any $n\geq 4$, we may take  $\Lambda_0= 4\varepsilon_1+4\varepsilon_2+\varepsilon_3+ \varepsilon_4$ and $\Lambda_0'= 4\varepsilon_1+3\varepsilon_2+ 3\varepsilon_3$, which satisfy  $\langle \dir,\Lambda_0\rangle = \langle \dir,\Lambda_0'\rangle= 4$ and $\lambda(C,\pi_{\Lambda_0}) = \lambda(C,\pi_{\Lambda_0'}) = 20n-4$, thus $\mathcal E(\dir,\Lambda_0)= \mathcal E(\dir,\Lambda_0')$ by Proposition~\ref{prop3:stringeigenvalues}. 
Moreover, for $\mu= b_1 \varepsilon_1+3\varepsilon_2+\varepsilon_3$ with $b_1\geq4$, $\Lambda_0$ and $\Lambda'_0  \in \PP_{\tau_\mu}$.
\end{example}

	The authors expect the existence of arbitrary large families of strings having the same Casimir eigenvalues and lying in  $\widehat G_\tau$ for some $\tau\in\widehat K$.
	Table~\ref{table:strings} shows many examples found with computer help that evidence this claim. 
	This information tells us that the study of the representation-spectral converse for $(\SO(2n),\SO(2n-1),\tau_\mu)$ for $\mu$ not satisfying the conditions in Theorem~\ref{thm4:formSO(2n)SO(2n-1)} require additional techniques that allow handling infinitely many non-tame eigenvalues.

\begin{remark}
The examples in the Table~\ref{table:strings} give more examples of $\tau \in \widehat K$ and strings in $\widehat G_\tau$ having the same set of Casimir eigenvalues. 
We note that, for any $\mu = \sum_{i=1}^{n-1} b_i\ee_i$ in the table, it follows immediately that the same set of strings are in $\widehat G_{\tau_{\mu+b\ee_1}}$ for any $b\in\N_0$. 
\end{remark}

\begin{table}
\caption{A list of strings in $\widehat G_{\tau_\mu}$ for $G=\SO(2n)$, with $\omega =\varepsilon_1$ with the same set of Casimir eigenvalues for any choice of $\Lambda$ in the fourth column and $\mu$  in the second column. 
We are abbreviating $\mu = b_1\ee_1+\dots+b_{n-1}\ee_{n-1}$ by $[b_1,\dots,b_{n-1}]$, $\Lambda = a_1\ee_1+\dots+a_{n}\ee_{n}$ by $[a_1,\dots,a_{n}]$ for the string bases $\Lambda$. 
The third column shows the number of string bases with the same Casimir eigenvalues. 
The strings' bases in the table were obtained with computer help. Hence, for these choices, there are infinitely many non tame eigenvalues in $\mathcal E_{\tau_\mu}$.  
}
\label{table:strings}

$
\begin{array}{cccl}
n& \mu&\# &\multicolumn{1}{c}{\text{strings bases} }
\\[1mm]
\hline \hline
3&[4,3] & 2&
	\begin{array}{l} [4, 3, 3] , [4, 4, 0]\end{array}
\\[1mm] \hline
3&[17,13]&3&
	\begin{array}{l}[17,14,10],[17,16,10], [17,17,1]\end{array}
\\ [1mm]\hline
3&[32,23]&4&
	\begin{array}{l} [32, 23, 23] , [32, 30, 12] , [32, 31, 9] , [32, 32, 4] \end{array}
\\[1mm] \hline
3&[64,50]&5&
	\begin{array}{l} [64, 51, 39] , [64, 55, 33] , [64, 59, 25] , [64, 62, 16] , [64, 64, 0] \end{array}
\\[1mm] \hline
3&[73,53]&6&
	\begin{array}{l} [73, 54, 50] , [73, 61, 41] , [73, 69, 25] , [73, 70, 22] , [73, 72, 14] , \\{} [73, 73, 7]  \end{array}
\\[1mm] \hline\hline
4&[4,3,1]&2&
	\begin{array}{l} [4, 3, 3, 0] , [4, 4, 1, 1]   \end{array}
\\[1mm] \hline
4&[8, 6, 2]&3&
	\begin{array}{l} [8, 6, 6, 2] , [8, 7, 5, 0] , [8, 8, 3, 1]   \end{array}
\\[1mm] \hline
4&[14, 12, 8]&4&
	\begin{array}{l} [14, 12, 8, 8] , [14, 12, 11, 1] , [14, 13, 9, 4] , [14, 14, 8, 2]  \end{array}
\\[1mm] \hline
4&[18,13,2]&5&
	\begin{array}{l} [18, 13, 13, 2] , [18, 14, 12, 0] , [18, 16, 9, 1] , [18, 17, 7, 0] , [18, 18, 4, 0]		  \end{array}
\\[1mm] \hline
4&[23, 19, 14]&6&
	\begin{array}{l} [23, 19, 15, 13] , [23, 19, 18, 8] , [23, 19, 19, 5] , [23, 21, 15, 9] ,\\ {}  [23, 22, 16, 1] , [23, 23, 14, 4] \end{array}
\\[1mm] \hline
4&[25, 22, 14]&7&
	\begin{array}{l} [25, 22, 15, 13] , [25, 22, 18, 8] , [25, 22, 19, 5] , [25, 24, 14, 10] ,\\ {}  [25, 24, 16, 6] , [25, 24, 17, 1] , [25, 25, 15, 4] \end{array}
\\[1mm] \hline
4&[35, 30, 22]&8&
	\begin{array}{l} [35, 30, 22, 21] , [35, 30, 30, 3] , [35, 31, 27, 11] , [35, 31, 28, 8] ,\\ {}  [35, 33, 24, 12] , [35, 34, 22, 13] , [35, 35, 23, 7] , [35, 35, 24, 0]  \end{array}
\\[1mm] \hline\hline
5&[4, 3, 1, 0]&2&
	\begin{array}{l} [4, 3, 3, 0, 0] , [4, 4, 1, 1, 0]\end{array}
\\[1mm] \hline
5&[6, 5, 4, 2]&3&
\begin{array}{l} [6, 5, 4, 4, 1] , [6, 5, 5, 2, 2] , [6, 6, 4, 2, 0] \end{array}
\\[1mm] \hline
5&[9, 8, 6, 2]&4&
	\begin{array}{l} [9, 8, 6, 6, 0] , [9, 8, 8, 2, 2] , [9, 9, 6, 4, 1] , [9, 9, 7, 2, 0] \end{array}
\\[1mm] \hline
5&[11, 10, 6, 2]&5&
\begin{array}{l} 
	[11, 10, 6, 6, 2] , [11, 10, 7, 5, 0] , [11, 10, 8, 3, 1] , [11, 11, 6, 4, 1] , \\ {} [11, 11, 7, 2, 0] 
\end{array}
\\[1mm] \hline
5&[15, 13, 9, 2]&6&
\begin{array}{l} 
	[15, 13, 9, 9, 1] , [15, 13, 11, 6, 2] , [15, 13, 12, 4, 1] , [15, 14, 9, 7, 2] , \\ {} [15, 14, 11, 3, 2] , [15, 15, 10, 2, 1] 
\end{array}
\\[1mm] \hline
5&[14, 12, 9, 2]&7&
\begin{array}{l} 
	[14, 12, 9, 9, 2] , [14, 12, 10, 8, 0] , [14, 12, 12, 4, 2] , [14, 13, 10, 6, 1] , \\ {} [14, 13, 11, 4, 0] , [14, 14, 9, 5, 2] , [14, 14, 10, 3, 1] 
\end{array}
\\[1mm] \hline
5&[21, 18, 13, 2]&8&
\begin{array}{l} 
	[21, 18, 13, 13, 2] , [21, 18, 14, 12, 0] , [21, 18, 16, 9, 1] , [21, 18, 17, 7, 0] ,\\ {} [21, 18, 18, 4, 0] , [21, 20, 14, 8, 0] , [21, 20, 16, 2, 2] , [21, 21, 13, 7, 1] 
\end{array}
\\[1mm] \hline
5&[21, 18, 16, 13]&9&
\begin{array}{l} 
	[21, 18, 16, 14, 12] , [21, 18, 18, 16, 2] , [21, 19, 17, 14, 8] , \\ {}[21, 19, 17, 16, 0] , [21, 19, 18, 14, 5] , [21, 20, 16, 15, 5] , \\ {}[21, 20, 18, 13, 3] , [21, 21, 16, 14, 3] , [21, 21, 17, 13, 1]
\end{array}
\\[1mm] \hline
5&[23, 20, 13, 2]&10&
\begin{array}{l} 
	[23, 20, 13, 13, 2] , [23, 20, 14, 12, 0] , [23, 20, 16, 9, 1] , [23, 20, 17, 7, 0] ,\\ {} [23, 20, 18, 4, 0] , [23, 21, 14, 10, 1] , [23, 21, 17, 3, 1] , [23, 22, 13, 9, 2] , \\ {}[23, 22, 15, 5, 2] , [23, 23, 13, 6, 2] 
\end{array}
\\[1mm] \hline\hline
\end{array}
$

\bigskip

\end{table}

\subsubsection*{Applications to $p$-spectra}
We next consider the representation $\tau_p$ (see Definition~\ref{def2:tau_p}) associated to the $p$-form spectrum. 
In the present case, $\tau_p$ is irreducible with highest weight $\varepsilon_1+\dots+\varepsilon_p$ for every $0\leq p\leq n-1$.  
The representation-spectral converse for $p$-forms holds over odd-dimensional spheres presented as $\Ot(2n)/\Ot(2n-1)$ (see for instance \cite[Thm.~1.2~(i)]{LMR-repequiv}).
The next results show that everything works in the same manner for $p<n-1$, but the situation changes drastically for $p=n-1$ when $n$ is even. 
The {next proposition} follows from Corollary~\ref{cor4:SO(2n)SO(2n-1)situations}~(i).

\begin{proposition}\label{thm4:p-isoSO(2n)}
The strong representation-spectral converse is valid for $(\SO(2n), \SO(2n-1) ,\tau_{p})$ for all $0\leq p\leq n-2$, and also for $p=n-1$ provided that $n$ is odd. 
\end{proposition}

\begin{theorem}\label{thm4:conversefailsSO(4n)}
For any $n\geq 2$, $n$ even, there exist $\Gamma$ and $\Gamma'$ finite subgroups of $G=\SO(2n)$ such that $\Gamma\ba S^{2n-1}$ and $\Gamma'\ba S^{2n-1}$ are isospectral on $(n-1)$-forms, but they are not $\tau_{n-1}$-representation equivalent in $G$. 
\end{theorem}

\begin{proof}
	Let $q\in \N$, $q>n$. The cyclic groups of order $q$, $\Gamma: = \{\diag(R(\tfrac{2\pi h}{q}), \dots, R(\tfrac{2\pi h}{q})): h\in\Z \}$, $\Gamma':= \{\diag(R(\tfrac{2\pi h}{q}), \dots, R(\tfrac{2\pi h}{q}), R(\tfrac{-2\pi h}{q})): h\in\Z \}$, are contained in the maximal torus $T$ of $G$ and they act freely on $S^{2n-1}$. The corresponding manifolds $\Gamma\ba S^{2n-1}$ and $\Gamma'\ba S^{2n-1}$ are  \emph{lens spaces}. 
	They are isometric since $\Gamma$ and $\Gamma'$ are conjugate in $\Ot(2n)$.
	Hence  $\Gamma\ba S^{2n-1}$ and $\Gamma'\ba S^{2n-1}$ are $(n-1)$-isospectral. 
	The rest of the proof is devoted to show that  $\Gamma$ and $\Gamma'$ are not $\tau_{n-1}$-representation equivalent in $\SO(2n)$.

	The irreducible representations of $\SO(2n)$ with highest weights $\Lambda_0:=\varepsilon_1+\dots+\varepsilon_{n}$ and $\overline{\Lambda}_0:=\varepsilon_1+\dots+\varepsilon_{n-1}-\varepsilon_n$ contain the $K$-type $\tau_{n-1}$ by \eqref{eq4:SO(2n)entrelazamiento}, i.e.\  they are in $\widehat G_{\tau_{n-1}}$.
	In what follows we shall abbreviate $\pi_{+}=
	\pi_{\Lambda_0}$, $\pi_{-}= \pi_{\overline \Lambda_0}$, and $V_\pm=V_{\pi_{\pm}}$. 
	We claim that 
	\begin{equation}\label{eq4:claimcountexSO(4n)}
	n_{\Gamma}(\pi_{\pm})-n_{\Gamma'}(\pi_{\pm})= \pm\binom{n}{n/2}\neq0,
	\end{equation}
	which will prove the assertion in the theorem.  
	It is important to recall that $\pi_{+}$ and $\pi_{-}$ are not conjugate to each other since $n$ is assumed to be even. 
	Indeed, every irreducible representation of $G=\SO(2n)$ for $n$ even is self-conjugate (see \cite[VI.(5.5)(ix)]{BrockerDieck}).

	Since $\Gamma, \Gamma'\subset T$, their elements  preserve the weight spaces in the  decompositions $V_\pm=\bigoplus_\eta V_\pm(\eta)$. 
	Here, $V_{\pm}(\eta)$ is the weight space associated to $\eta$, that is, $V_{\pm}(\eta)=\{v\in V_{\pm}: \pi_\pm (\exp(X))\cdot v= e^{\eta(X)}\, v\;\text{for all} X\in\mathfrak t\}$.
	Consequently, 
	$$
	n_\Gamma(\pi_\pm) = \dim V_{\pm}^\Gamma= \sum_{\eta} \dim V_{\pm}(\eta)^\Gamma
	= \sum_{\eta:\, V_{\pm}(\eta)^\Gamma=V_{\pm}(\eta)} \dim V_{\pm}(\eta)
	$$ 
	and similarly for $\Gamma'$. 
	The multiplicity of $\eta$ in $\pi_{\pm}$, $\dim V_{\pm}(\eta)$, has been explicitly computed (see for instance \cite{LR-fundstring}).
	Indeed, an arbitrary element $t=\diag \left(R(\theta_1),\dots,R(\theta_n)\right)\in T$ acts on a weight space of weight $\eta= \sum_{j=1}^n a_j\varepsilon_j\in \PP(G)$ by the scalar $e^{-i\sum_{j=1}^n a_j\theta_j}$. 
	Hence, $V_\pm(\eta)$ is invariant by $\Gamma$ (resp.\ $\Gamma'$) if and only if $\sum_{j=1}^n a_j \equiv0\pmod q$ (resp.\ $\sum_{j=1}^{n-1} a_j-a_n \equiv0\pmod q$). 
	The congruence equalities can be replaced by equalities since $q>n$ and, by \cite[(16)]{LR-fundstring}, $V_\pm(\mu)\neq0$ forces $|a_j|\leq 1$ for all $j$.
	One also has that $\norma{\eta}:=\sum_{j=1}^n |a_j|=n-2r$ for some non-negative integer $r$. 
	
	We now split $V_\pm$ as 
	$V_{\pm} = W_{\pm} \oplus U_{\pm}$, where 
	\begin{align*}
		W_{\pm}&= \sum_{\eta: \norma{\eta}=n} V_{\pm}(\eta),&
		U_{\pm}&= \sum_{\eta: \norma{\eta}<n} V_{\pm}(\eta). 
	\end{align*}
	By \cite[Rem.~IV.5]{LR-fundstring}, for $\eta=\sum_{j=1}^n a_j\varepsilon_j$  and $|a_j|\leq 1$ for all $j$, one has that $\dim V_+(\eta)= \dim V_-(\eta) = \binom{2r}{r}$, where $r=\frac{n-\norma{\eta}}{2}\in\N$. 
	We conclude that $U_+=U_-$, thus $\dim U_{+}^{\Gamma}= \dim U_{-}^{\Gamma}$ and $\dim U_{+}^{\Gamma'}= \dim U_{-}^{\Gamma'}$.
	For $\eta$ as above, a (non-trivial) weight space of weight $\eta$ is invariant by $\Gamma$ if and only if the (non-trivial) weight space of weight $\overline\eta$ is invariant by $\Gamma'$, where $\overline\eta=\sum_{j=1}^{n-1}a_j\varepsilon_j - a_n\varepsilon_n$. 
	It hence follows that $\dim U_{+}^{\Gamma}= \dim U_{-}^{\Gamma}=  \dim U_{+}^{\Gamma'}= \dim U_{-}^{\Gamma'}$. 
	
	We have shown so far that $n_{\Gamma}(\pi_\pm) - n_{\Gamma'}(\pi_{\pm}) = \dim W_{\pm}^{\Gamma}- \dim W_{\pm}^{\Gamma'}$. 
	By \cite[Rem.~IV.5]{LR-fundstring}, if $\eta=\sum_{j=1}^n a_j\varepsilon_j$ with $\norma{\eta}=n$ and $|a_j|=1$ for all $j$, then $\dim V_+(\eta)=1$ and $\dim V_-(\eta) = 0$ if $\#\{j: 1\leq j\leq n,\; a_j=-1\}$ is even. 
	Otherwise $\dim V_+(\eta)=0$ and $\dim V_-(\eta) = 1$. 
	Hence, $\dim W_{+}^{\Gamma}=\binom{n}{n/2}=\dim W_{-}^{\Gamma'}$, $\dim W_{-}^{\Gamma}=0=\dim W_{+}^{\Gamma'}$, which proves \eqref{eq4:claimcountexSO(4n)} and completes the proof of the theorem.
\end{proof}

We next explain the above counter-example by using outer automorphisms of $G=\SO(2n)$.

\begin{remark} \label{rem4:outer-automorphism}
Let $g_0=\diag(1,\dots,1,-1)\in\Ot(2n)$. 
Although $g_0\notin G$, the map $\varphi:G\to G$ given by $\varphi(x)=g_0^{-1}xg_0=g_0xg_0$ is an automorphism of $G$.
One can check that $\varphi$ is not an inner automorphism and $\varphi^2$ is the identity map on $G$. 

For $\Gamma$ and $\Gamma'$ as in the proof of Theorem~\ref{thm4:conversefailsSO(4n)}, one has that $\varphi(\Gamma)=\Gamma'$, thus $\Gamma\ba S^{2n-1}$ and $\Gamma'\ba S^{2n-1}$ are isometric because $\Gamma$ and $\Gamma'$ are conjugate in $\Ot(2n)$, and then $\Gamma\ba S^{2n-1}$ and $\Gamma'\ba S^{2n-1}$ are $\tau$-isospectral for all $\tau$.

For $\Lambda=\sum_{i=1}^n a_i\varepsilon_i\in \PP^{+}(G)$, we write $\bar\Lambda= \sum_{i=1}^{n-1} a_i\varepsilon_i-a_n\varepsilon_n$, which is again a dominant $G$-integral weight.
One has that $\pi_\Lambda \circ\varphi\simeq \pi_{\bar\Lambda}$ as representations of $G$ for any $\Lambda\in\PP^{+}(G)$ (see \cite[\S{}5.5.5]{GoodmanWallach-bookSpringer}).  
In other words, the representation $\pi_{\Lambda}\circ\varphi$ of $G$ is irreducible and has highest weight $\bar\Lambda$, thus $\pi_{\bar\Lambda}\circ\varphi\simeq \pi_\Lambda$. 
It then follows that $\dim V_{\pi_\Lambda}^{\Gamma'}= \dim V_{\pi_\Lambda}^{\varphi(\Gamma)} = \dim V_{\pi_\Lambda\circ\varphi}^{\Gamma} = \dim V_{\pi_{\bar\Lambda}}^{\Gamma}$ and similarly $\dim V_{\pi_\Lambda}^{\Gamma}= \dim V_{\pi_{\bar\Lambda}}^{\Gamma'}$, which yields
\begin{equation}
\dim V_{\pi_\Lambda}^{\Gamma}+\dim V_{\pi_{\bar\Lambda}}^{\Gamma}= 
\dim V_{\pi_{\Lambda}}^{\Gamma'}+ \dim V_{\pi_{\bar\Lambda}}^{\Gamma'}.
\end{equation} 
For arbitrary $\Gamma,\Gamma$ finite subgroups of $G$ and for all $\pi_\Lambda\in\widehat G_{\tau_{n-1}}$, the above identity is equivalent to {the fact that} $\Gamma$ and $\Gamma'$ are $\widetilde \tau_{n-1}$-representation equivalent (established in \cite[Cor.~1.2~(i)]{LMR-repequiv}), where $S^{2n-1}$ is realized as $\Ot(2n)/\Ot(2n-1)$ and $\widetilde\tau_{n-1}$ is the irreducible representation introduced in Definition~\ref{def2:tau_p} for $K=\Ot(2n-1)$ .

The difficulty in the proof of Theorem~\ref{thm4:conversefailsSO(4n)} was to show that $\dim V_{\pi_{\Lambda}}^\Gamma \neq \dim V_{\pi_{\Lambda}}^{\Gamma'}$ for at least one $\pi_{\Lambda}\in \widehat G_{\tau_{n-1}}$. 
This is never the case when $n$ is odd because $\pi_{\Lambda}$ and $\pi_{\bar\Lambda}$ are dual to each other (see \cite[VI.(5.5)(ix)]{BrockerDieck}), and consequently $\dim V_{\pi_\Lambda}^{\widetilde\Gamma} = \dim V_{\pi_{\bar \Lambda}}^{\widetilde\Gamma}$ for every finite subgroup $\widetilde\Gamma$ of $G$, by Remark~\ref{rem2:Wallach}. 

We conclude the remark by noting that, although there are other complex simple Lie algebras having a non-trivial outer automorphism $\varphi$, one always has in theses cases that $\pi\circ\varphi \simeq \pi^*$ for all $\pi \in \widehat G$.  
Therefore, one cannot construct similar counterexamples in any other compact symmetric space $G/K$ of real rank one. 
\end{remark}

\begin{remark}\label{rem4:conversefailsSO(4n)}
If $n$ is even, similar examples as in Theorem~\ref{thm4:conversefailsSO(4n)} should exist for any $\tau_\mu$ with $\mu=\sum_{i=1}^{n-1} b_i\varepsilon_i$ such that $b_{n-1}>0$. 
Consequently, the condition $b_{n-1}=0$ could not possibly be avoided in the assumptions of Theorem~\ref{thm4:formSO(2n)SO(2n-1)} (see Remark~\ref{rem4:b_n-1=0}).
\end{remark}

\begin{remark}
By realizing the $7$-dimensional round sphere as the normal homogeneous space $\Spin(8)/\Spin(7)$ endowed with the $\Ad(\Spin(8))$-invariant inner product $\langle X,Y\rangle = -\tfrac12\tr(XY)$, for $X,Y\in\mathfrak{spin}(8)=\so(8)$, we have at hand the triality phenomenon (see for instance \cite[\S{}20.3]{FultonHarris-book}).
The outer automorphism group is isomorphic to the symmetric group on three elements. 
Its elements permute the three $8$-dimensional representations of $\Spin(8)$: the standard representation $\pi_{\varepsilon_{1}}$ and the spin representations $\pi_{\frac12(\varepsilon_{1}+ \varepsilon_{2}+\varepsilon_{3}-\varepsilon_{4})}$ and $\pi_{\frac12(\varepsilon_{1}+ \varepsilon_{2}+\varepsilon_{3}+\varepsilon_{4})}$. 
This situation should provide more counterexamples to the representation-spectral converse as in Theorem~\ref{thm4:conversefailsSO(4n)}, according to the observations in Remark~\ref{rem4:outer-automorphism}.
\end{remark}

\begin{remark}
It is important to note that the notion of $\tau$-representation equivalence in $G$ depends heavily on the choice of $(G,K)$, and not only on the quotient $G/K$. 
For example, Theorem~\ref{thm4:conversefailsSO(4n)} shows that the representation-spectral converse fails for $(\SO(2n),\SO(2n-1),\tau_{n-1})$, but it is valid for $(G,K)=(\Ot(2n),\Ot(2n),\tau_{n-1})$ by \cite[Thm.~1.2~(i)]{LMR-repequiv}. 
Actually, {we} do not  know if there is a finite-dimensional representation $\tau$ of $\Ot(2n-1)$ such that the converse fails.
\end{remark}

\subsubsection*{Three-dimensional spherical space forms} 
We conclude this section with a different kind of application of the strong multiplicity one theorem for strings. 

\begin{proposition}\label{prop4:3-dim}
Let $\Gamma$ and $\Gamma'$ be finite subgroups of $\SO(4)$ and let $\tau$ be an irreducible representation of $\SO(3)$. 
If $\Gamma\ba S^3$ and $\Gamma'\ba S^3$ are $\tau$-isospectral, then they are $0$-isospectral. 
\end{proposition}

\begin{proof}
Let $\mu=b\varepsilon_1$ be the highest weight of $\tau$.
Let $\tau_0$ denote the trivial representation of $\SO(3)$, which has highest weight $0$. 
By \eqref{eq4:P_tauSO(2n)}, we have that $\PP_{\tau_0}=\{0\}$ and $\PP_{\tau}= \{b\varepsilon_1+a_2\varepsilon_2: a_2\in\Z,\; |a_2|\leq b\}$. 
Observe that Theorem~\ref{thm4:formSO(2n)SO(2n-1)} is not available because $n=2$ is even and $b>0$. 
However, the proof of Theorem~\ref{thm3:tauiso=>tauequiv} in this particular case shows that $n_{\Gamma}(\pi_{k\varepsilon_1}) = n_{\Gamma'}(\pi_{k\varepsilon_1})$ for all but finitely many $k\geq b$.
In fact, from Lemma~\ref{lem3:coincidences} we conclude that the polynomial $\lambda(C,\pi_{(k+b) \varepsilon_1})= (k+b)(k+b+2)$ has finitely many coincidences with $\lambda(C,\pi_{(k+b) \varepsilon_1+a_2\varepsilon_2})= (k+b)(k+b+2)+a_2^2$ for any $a_2\neq 0$.
Now, Corollary~3.2 in \cite{LM-strongmultonethm} implies that  $n_{\Gamma}(\pi_{k\varepsilon_1}) = n_{\Gamma'}(\pi_{k\varepsilon_1})$ for all $k\geq0$, that is, $\Gamma$ and $\Gamma'$ are $\tau_0$-representation equivalent. 
Consequently, $\Gamma\ba S^3$ and $\Gamma'\ba S^3$ are $0$-isospectral, as asserted. 
\end{proof}

As a consequence, we  extend a result of Ikeda (\cite[Thm.~I]{Ikeda80_3-dimI}) asserting  that  if two $3$-dimensional spherical space forms are isospectral with respect to the Laplace--Beltrami operator (i.e.\ $0$-isospectral), then they are isometric.

\begin{theorem}\label{thm4:tau-iso3-dim}
If two $3$-dimensional spherical space forms are $\tau$-isospectral for any irreducible representation $\tau$ of $\SO(3)$, then they are isometric. 
\end{theorem}

\begin{remark}
Naveed Shams Ul Bari and Eugenie Hunsicker~\cite{ShamsHunsicker17} recently proved that two  $3$-dimensional lens orbifolds (i.e.\ $\Gamma\ba S^3$ with $\Gamma$ a cyclic subgroup of $\SO(4)$) are $0$-isospectral if and only if they are isometric. 
Consequently, Proposition~\ref{prop4:3-dim} implies that two $\tau$-isospectral $3$-dimensional lens orbifolds are necessarily isometric. 
For a recent account on $p$-isospectrality among lens spaces, see \cite{LMR-SaoPaulo}.
\end{remark}

\section{Even dimensional spheres}
\label{sec:evenspheres}
We now consider even-dimensional spheres. 
This case has many similarities with the  previous one, so we will omit many details. 
Throughout the section, for any $n\geq2$, we set $G=\SO(2n+1)$ and let $K$ be the subgroup isomorphic to $\SO(2n)$ embedded in the upper left-hand block of $G$. 
Hence $G/K$ is diffeomorphic to $S^{2n}$. 
We fix the metric induced by the $\op{Ad}(G)$-invariant inner product $\langle X,Y\rangle =-\frac12 \tr(XY)$ on $\mathfrak g$ which gives constant sectional curvature one to $S^{2n}$.

We pick the maximal torus in $G$ given by 
$
T=\{\diag (R(\theta_1),\dots,R(\theta_{n}), 1): \theta_j\in\R\}, 
$
with $R(\theta)$ as defined in the previous section. 
Note that $T=T\cap K$ is also a maximal torus in $K$.
We have already described the root systems $\Phi(\mathfrak g_\C,\mathfrak t_\C)$ and $\Phi(\mathfrak k_\C,\mathfrak t_\C)$ in the previous section. 
We will only recall that $\{\varepsilon_1,\dots,\varepsilon_n\}$ is an orthornormal basis of $\mathfrak t_\C^*$ with respect to $\langle\cdot,\cdot\rangle$, $\PP^{+}(G) = \{\sum_{j=1}^n a_j\varepsilon_j\in \PP(G): a_1\geq\dots \geq a_n\geq0\}$, and $\PP^{+}(K) = \{\sum_{j=1}^{n} b_j\varepsilon_j\in \bigoplus_{j=1}^{n}\Z\varepsilon_j: b_1\geq\dots\geq b_{n-1}\geq |b_n|\}$.

We recall the well-known branching law {in} this case (see for instance \cite[Thm.~8.1.3]{GoodmanWallach-bookSpringer} and \cite[Thm.~9.16]{Knapp-book-beyond}).
If $\tau_\mu \in\widehat K$ and $\pi_\Lambda \in\widehat G$ have highest weights $\mu=\sum_{j=1}^{n} b_j\varepsilon_j\in\PP^{+}(K)$, $\Lambda= \sum_{j=1}^{n} a_j \varepsilon_j \in\PP^{+}(G)$ respectively, then $\tau_\mu$ occurs in the decomposition of $\pi_\Lambda|_K$ if and only if
\begin{equation}\label{eq5:SO(2n+1)entrelazamiento}
a_1\geq b_1\geq a_2\geq b_2\geq \dots \geq a_{n}\geq |b_{n}|.
\end{equation}
Furthermore, in this case  $\dim \Hom_K(\tau_\mu,\pi_{\Lambda})=1$.

\begin{lemma}\label{lem5:hatG_tauSO(2n+1)}
For $\mu=\sum_{j=1}^{n} b_j\varepsilon_j\in \PP^{+}(K)$, we have that 
\begin{equation}\label{eq5:hatSO(2n+1)_tau}
\widehat G_{\tau_\mu} 
= \bigcup_{\Lambda_0\in \PP_{\tau_\mu}} \String{\varepsilon_1,\Lambda_0}
= \bigcup_{\Lambda_0\in \PP_{\tau_\mu}} \{\pi_ {\Lambda_0+k\varepsilon_1}:k\in \N_0\}.
\end{equation}
where
\begin{equation}
\PP_{\tau_\mu} = \left\{ \sum_{j=1}^{n} a_j\varepsilon_j :a_1=b_1,\, a_j\in\Z \, \textrm{ for every } j,\,   \eqref{eq5:SO(2n+1)entrelazamiento} \text{ holds}\right\}.
\end{equation}
Thus, $\widehat G_{\tau_\mu}$ is a finite disjoint union of strings with direction $\dir = \varepsilon_1$.
\end{lemma}

We now apply Theorem~\ref{thm3:tauiso=>tauequiv} in the present case.
Again, it turns out that condition (ii) in Proposition~\ref{prop3:stringeigenvalues} is never satisfied, just like  \eqref{eq3:dual} since every irreducible representation of $\SO(2n-1)$ is self-conjugate (see for instance \cite[VI.(5.6)]{BrockerDieck}). 

Consequently, we need only check whether $\lambda(C,\Lambda_0)\neq  \lambda(C,\Lambda_0')$ for all $\Lambda_0\neq \Lambda_0'\in \mathcal P_{\tau_{\mu}}$. 
Let $\Lambda_0=\sum_{i=1}^{n} a_i \varepsilon_i\in \PP_{\tau_\mu}$.
By \eqref{eq3:lambda(C,pi)}, since $\rho_G=\sum_{i=1}^{n}(n-i+1/2)\varepsilon_i$, \begin{align}
	\lambda(C,\pi_{\Lambda_0}) 
	= \langle \Lambda_0,\Lambda_0 +2\rho_G\rangle 
	= b_1(b_1+1+2(n-1)) + \sum_{i=2}^n a_i(a_i+1+2(n-i)).
\end{align}
We thus have the following result. 

\begin{theorem}\label{thm5:formSO(2n+1)SO(2n)}
	Let $G=\SO(2n+1)$, $K=\SO(2n)$, and $\tau_\mu\in\widehat K$ with $\mu= \sum_{i=1}^{n} b_i\varepsilon_i\in \PP^{+}(K)$. 
	Assume the form 
	\begin{equation}\label{eq5:formSO(2n+1)SO(2n)}
	(a_2,\dots,a_n)\longmapsto \sum_{i=2}^{n} a_i(a_i+1+2(n-i)) 
	\end{equation}
	represents different numbers on the set 
	\begin{equation}\label{eq5:T_GmuSO(2n+1)}
	\TT_{G,\mu} :=\{(a_2,\dots,a_n) \in\Z^{n-1}: b_1\geq a_2\geq b_2\geq \dots\geq b_{n-1}\geq a_{n}\geq |b_n|\}.
	\end{equation}  
	Then, a finite part of the spectrum $\mathcal E_\tau$ as in Theorem~\ref{thm3:tauiso=>tauequiv},  determines the whole spectrum of $\Delta_{\tau,\Gamma}$ for any finite subgroup $\Gamma$ of $G$.
	Furthermore, the strong representation-spectral converse is valid for $(\SO(2n+1),\SO(2n),\tau_\mu)$.
\end{theorem}

The next corollary gives infinitely many $K$-types $\tau_\mu$ for which the assumption in the theorem is valid. 
In this case, the representation-spectral converse was known to hold only for the trivial representation of $K$ (see \cite[Prop.~3.2]{Pesce96}). 
The proof is straightforward and left to the reader.

\begin{corollary}\label{cor5:SO(2n+1)SO(2n)situations}
Let $G=\SO(2n+1)$, $K=\SO(2n)$, and $\tau_\mu\in\widehat K$ with $\mu= \sum_{i=1}^{n} b_i\varepsilon_i\in \PP^{+}(K)$.
Assume any of the following conditions:
\begin{enumerate}\renewcommand{\labelenumi}{(\roman{enumi})}
\item\label{item5:b_1<=3} $b_1-|b_n|\leq 3$; 
		
\item\label{item5:onejump} there is $2\leq j\leq n$ such that $b_i-b_{i+1}=0$ for all $1\leq i\leq n-1$, $i\neq j$;

\item\label{item5:twojumps} $b_2-|b_n|\leq 2$, $b_1$ arbitrary.
\end{enumerate}
Then, the same consequences as in Theorem~\ref{thm4:formSO(2n)SO(2n-1)} hold, in particular, the strong representation-spectral converse is valid for $(\SO(2n),\SO(2n-1),\tau_\mu)$.
\end{corollary}

Similarly to Example~\ref{ex4:SO(2n)unknown-cases}, one can exhibit two disjoint strings having the same eigenvalues, showing that condition (i) in Corollary~\ref{cor5:SO(2n+1)SO(2n)situations} is optimal.

\begin{example}\label{ex5:unknown-casesSO(2n+1)}
For $n=3$, let $\Lambda_0=4\varepsilon_1+4\varepsilon_2+\varepsilon_3$ and $\Lambda_0'=4\varepsilon_1+3\varepsilon_2+3\varepsilon_3$, thus $\langle \dir,\Lambda_0\rangle = \langle \dir,\Lambda_0'\rangle= 4$ and $\lambda(C,\pi_{\Lambda_0}) = \lambda(C,\pi_{\Lambda_0'}) = 66$, which gives $\mathcal E(\dir,\lambda_0) = \mathcal E(\dir,\lambda_0')$ by Proposition~\ref{prop3:stringeigenvalues}.

For any $n\geq 4$, $\Lambda_0= 4\varepsilon_1+4\varepsilon_2+\varepsilon_3+ \varepsilon_4$ and $\Lambda_0'= 4\varepsilon_1+3\varepsilon_2+ 3\varepsilon_3$ satisfy $\langle \dir,\Lambda_0\rangle = \langle \dir,\Lambda_0'\rangle= 4$ and $\lambda(C,\pi_{\Lambda_0}) = \lambda(C,\pi_{\Lambda_0'}) = 20n-6$, thus $\mathcal E(\dir,\lambda_0) = \mathcal E(\dir,\lambda_0')$ by Proposition~\ref{prop3:stringeigenvalues}.
	
Moreover, for $\mu= 4\varepsilon_1+3\varepsilon_2$ when $n=3$ and $\mu= 4 \varepsilon_1+3\varepsilon_2+\varepsilon_3$ when $n\geq4$, we have that $\Lambda_0,\Lambda_0' \in \PP_{\tau_\mu}$. 
Consequently, the strong representation-spectral converse for $(\SO(2n+1),\SO(2n),\tau_\mu)$ does not follows from Theorem~\ref{thm3:tauiso=>tauequiv}.
\end{example}

Similarly as shown in Table~\ref{table:strings}, there exist many examples of families of strings having the same Casimir eigenvalues.

\subsubsection*{Applications to $p$-spectra}
We now study the {case of} $p$-forms by considering the representation $\tau_p$ (see Definition~\ref{def2:tau_p}).
Then $\tau_p$ is irreducible with highest weight $\varepsilon_1+\dots+\varepsilon_p$ for every $0\leq p\leq n-1$, and $\tau_n\simeq \tau_n^+\oplus \tau_n^-$, where $\tau_n^\pm$ is the irreducible representation of $K\simeq \SO(2n)$ with highest weight $\varepsilon_1+\dots+\varepsilon_{n-1}\pm \varepsilon_n$.

\begin{theorem}\label{thm5:p-isoSO(2n+1)}
The strong representation-spectral converse is valid for $(\SO(2n+1), \SO(2n) ,\tau_{p})$ for all $p$.
\end{theorem}

\begin{proof}
The {assertion for} $0\leq p\leq n-1$ follows from Corollary~\ref{cor5:SO(2n+1)SO(2n)situations}. 
We thus assume $p=n$. 
Lemma~\ref{lem5:hatG_tauSO(2n+1)} implies $\widehat G_{\tau_n}=\widehat G_{\tau_n^+} = \widehat G_{\tau_n^-}$. 
It follows immediately that condition (i) in Theorem~\ref{thm3:tauiso=>tauequiv} holds for $\tau_n$ since it holds for $\tau_n^\pm$ by Corollary~\ref{cor5:SO(2n+1)SO(2n)situations}. 
\end{proof}

\section{Complex projective spaces} 
\label{sec:SU}
We now consider complex projective spaces. 
Throughout the section, for any $n\geq2$, we set $G=\SU(n+1)$ and 
$$
K =\left\{ \begin{pmatrix} g \\ & z\end{pmatrix}: g\in \U(n),\, z\in\U(1),\, \det(g)z=1 \right\}  \simeq \op{S}(\U(n)\times\U(1)). 
$$
Thus, $G/K$ is diffeomorphic to $P^n(\C)$.
We consider the metric induced by the inner product on $\mathfrak g$ given by $\langle X,Y\rangle = -\tr(XY)$, which gives sectional curvature $K$ satisfying $1\leq K\leq 4$.

We fix the maximal torus in $G$ given by 
\begin{equation*}
T= \{\diag(e^{i\theta_1},\dots,e^{i\theta_{n+1}}): \theta_j\in\R,\; \textstyle\sum_{j=1}^{n+1} \theta_j=0\}. 
\end{equation*}
The associated Cartan subalgebra $\mathfrak t_{\C}$ of $\mathfrak g_\C$ is a subspace of codimension one of 
$$
\mathfrak u_\C:= \{\diag(\theta_1,\dots,\theta_{n+1}):\theta_j\in\C\}. 
$$
We {let} $\varepsilon_j\in\mathfrak u_\C^*$ {be} given by 
$
\varepsilon_j(\diag(\theta_1,\dots,\theta_{n+1})) = \theta_j, 
$
for any $1\leq j\leq n+1$.
One has that 
$$
\mathfrak t_\C^* = \{\textstyle \sum_{j=1}^{n+1} a_j\varepsilon_j: a_j\in\C,\, \sum_{j=1}^{n+1}a_j=0\}. 
$$
It turns out that $\Phi(\mathfrak g_\C,\mathfrak t_\C)= \{\pm (\varepsilon_i-\varepsilon_j): 1\leq i<j\leq n+1\}$ and 
$$
\PP(G)= \{\textstyle \sum_{j=1}^{n+1} a_j\varepsilon_j: a_j-a_{j+1}\in\Z,\;  \sum_{j=1}^{n+1}a_j=0\}.
$$ 
For any $\sum_{j=1}^{n+1} a_j\varepsilon_j\in \PP(G)$, it follows that $(n+1)a_j\in\Z$ for every $j$. 
We extend $\langle\cdot,\cdot\rangle$ to $\mathfrak u_\C$ (and  to  its dual) by $\langle X,Y\rangle = -\tr(XY)$. 
Thus $\langle \varepsilon_i,\varepsilon_j\rangle =\delta_{i,j}$ for all $1\leq i,j\leq n+1$.
{For} the standard order, one has  the simple roots $\{\varepsilon_1-\varepsilon_{2},\dots,\varepsilon_{n}-\varepsilon_{n+1}\}$, $\Phi^+(\mathfrak g_\C,\mathfrak t_\C)= \{\varepsilon_i-\varepsilon_j: 1\leq i<j\leq n+1\}$, and furthermore 
$$
\PP^{+}(G) = \{\textstyle \sum_{j=1}^{n+1} a_j\varepsilon_j\in \PP(G): a_1\geq\dots\geq a_{n+1}\}. 
$$

The group $K$ is reductive with $1$-dimensional center $Z(K)= \{\diag(e^{i\theta}, \dots, e^{i\theta},e^{-in\theta}):\theta\in\R\}$. 
We note that $T$ is, as well, a maximal torus of $K$. 
In this case we have that 
$\Phi(\mathfrak k_\C,\mathfrak t_\C)= \{\pm (\varepsilon_i-\varepsilon_j): 1\leq i<j\leq n\}$ and $\PP(K)=\PP(G)$. 
We pick on $\mathfrak t_\C^*$ the same order as above, so the simple roots are $\{\varepsilon_1-\varepsilon_{2},\dots,\varepsilon_{n-1}-\varepsilon_{n}  \}$, $\Phi^+(\mathfrak k_\C,\mathfrak t_\C)= \{\varepsilon_i-\varepsilon_j: 1\leq i<j\leq n\}$, and 
$$
\PP^{+}(K)= \{\textstyle \sum_{j=1}^{n+1} b_j\varepsilon_j\in \PP(K)=\PP(G): b_1 \geq\dots\geq b_n\}.
$$ 

Here, we introduce a tool that facilitates the parametrization of elements in $\PP(G)$.
We define the projection $\pr$ from  $\mathfrak u_\R^*= \op{span}_\R\{\varepsilon_1,\dots, \varepsilon_{n+1}\} $ to $\mathfrak t_\R^*= \{\sum_{j=1}^{n+1} a_j \varepsilon_j\in \mathfrak u_\R^+: \sum_{j=1}^{n+1} a_j=0\}$
given by 
$$
 \pr\left(\sum_{j=1}^{n+1} a_j\varepsilon_j\right) = \sum_{j=1}^{n+1} a_j\varepsilon_j-\frac{1}{n+1} \left(\sum_{j=1}^{n+1}a_j\right) (\varepsilon_1+\dots+\varepsilon_{n+1}).
$$ 
Note that $\pr(\bigoplus_j \Z \varepsilon_j)=  \PP(G)$. 
For example, the standard representation of $G$ has highest weight $\pr(-\varepsilon_{n+1})$, its contragradient representation has highest weight  $\pr(\varepsilon_{1})$, and the $p$-th fundamental weight has highest weight $\pr(\varepsilon_1+\dots+ \varepsilon_p)$ for any $1\leq p\leq n$.

We now recall the branching law in the present case.  
The reader may see equivalent statements in \cite[\S5]{IkedaTaniguchi78} and \cite[\S3]{Halima07}, where {they}  utilize non-standard parameterizations of $\PP^{+}(G)$ and $\PP^{+}(K)$. 
We use exactly the same as in \cite[page 13]{Camporesi05Pacific}.

\begin{lemma}\label{lem6:SUbranchinglaw}
Let $G=\SU(n+1)$ and $K=\op{S}(\U(n)\times\U(1))$. 
If $\Lambda=\sum_{j=1}^{n+1} a_j\varepsilon_j\in \PP^{+}(G)$ and $\mu=\sum_{j=1}^{n+1} b_j\varepsilon_j\in \PP^{+}(K)$, then $\tau_\mu$ occurs in  $\pi_\Lambda|_K$ if and only if $a_{1}-b_1\in \Z$ and 
\begin{equation}\label{eq6:SUentrelazamiento}
a_1\geq b_1\geq a_2\geq b_2\geq \dots \geq a_{n}\geq b_{n}\geq a_{n+1}.
\end{equation} 
Furthermore, if this is the case,  $ \dim \Hom_K(\tau_\mu,\pi_\Lambda|_K)=1$. 
\end{lemma}

The next goal is to show that $\widehat G_{\tau}$ is a finite and disjoint union of strings with direction $\dir := \varepsilon_1-\varepsilon_{n+1}$ for any finite-dimensional representation $\tau$ of $K$.
See \cite[\S{}3]{Camporesi05Pacific} for the same result written in a different way.
Let $\mu=\sum_{j=1}^{n+1} b_j\varepsilon_j\in \PP^{+}(K)$, thus $\sum_{j=1}^{n+1}b_j=0$, $(n+1)b_j\in\Z$ for all $j$ and $b_1\geq\dots\geq b_n$. 
We set 
\begin{equation}\label{eq6:P_tauSU(n+1)}
\mathcal P_{\tau_\mu} = \left\{ 
\sum_{j=1}^{n+1} a_j\varepsilon_j: 
\begin{array}{l}
a_j-b_1\in\Z\text{ for all $2\leq j\leq n$},\\
b_1\geq a_2\geq b_2\geq \dots \geq a_n\geq b_n,\\
a_1=b_1+\max(0,-r),\\ a_{n+1}=b_n-\max(0,r), \text{where} \\
 r= b_1+b_n+\sum_{j=2}^{n} a_j
\end{array}
\right\}.
\end{equation}
It is a simple matter to check that $\mathcal P_{\tau_\mu}$ is finite and included in $\PP^{+}(G)$. 

\begin{lemma}\label{lem6:hatG_tauSU(n+1)}
For $\mu=\sum_{j=1}^{n+1} b_j\varepsilon_j\in \PP^{+}(K)$, we have that 
$$
\widehat G_{\tau_\mu} = \bigcup_{\Lambda_0\in \mathcal P_{\tau_\mu}} \String{\Lambda_0,\varepsilon_1-\varepsilon_{n+1}} = 
\bigcup_{\Lambda_0\in \mathcal P_{\tau_\mu}} \{\pi_{\Lambda_0+k(\varepsilon_1-\varepsilon_{n+1})}: k\geq0\}. 
$$
{Furthermore,} the union is finite and disjoint.
\end{lemma}

\begin{proof}
By using Lemma~\ref{lem6:SUbranchinglaw}, it is a simple matter to check that $\Lambda_0 \in \PP_{\tau_\mu}$ {implies that} $\pi_{k\dir+\Lambda_0}\in \widehat G_{\tau_\mu}$ for all $k\geq0$. 
To see the converse, suppose $\pi_{\Lambda}\in \widehat G_{\tau_\mu}$ for some $\Lambda=\sum_{i=1}^{n+1} a_i\varepsilon_i \in \PP^{+}(G)$, thus $a_i-b_j\in\Z$ for all $i,j$ and \eqref{eq6:SUentrelazamiento} holds. 
We set $k=\min(a_1-b_1,b_n-a_{n+1})$. 
It follows that 
$$
\Lambda_0:= \Lambda - k\dir = \Lambda - k(\varepsilon_1-\varepsilon_{n+1}) = (a_1-k)\varepsilon_1+\sum_{i=2}^{n} a_i\varepsilon_i+(a_{n+1}+k)\varepsilon_{n+1}
$$ 
still satisfies \eqref{eq6:SUentrelazamiento}, thus $\pi_{\Lambda_0}\in \widehat G_{\tau_\mu}$. 
It suffices to show that $\Lambda_0 \in \PP_{\tau_\mu}$. 
The first two conditions in \eqref{eq6:P_tauSU(n+1)} are clearly satisfied.
To check the last two conditions, set $r= b_1+b_n+\sum_{j=2}^{n} a_j$.
Thus $r = b_1+b_n-a_1-a_{n+1}$ since $\sum_{i=1}^{n+1}a_i=0$. 
Now, if $a_1-b_1\leq b_n-a_{n+1}$, then $k=a_1-b_1$, $r\geq 0$, $a_1-k=b_1=b_1+\max(0,-r)$, and $a_{n+1}+k=a_{n+1}+a_1-b_1=b_n-r=b_n-\max(0,r)$, which shows that $\Lambda_0\in \PP_{\tau_\mu}$. 
The case $a_1-b_1> b_n-a_{n+1}$ is completely analogous. 
\end{proof}

We now give  particular situations when Theorem~\ref{thm3:tauiso=>tauequiv} can be applied. 
The next result considers the case of two jumps of length  one and also {of} one jump of length two {within} the first $n$ coefficients of $\mu$, with a few technical exceptions. 
The proof  reduces to showing that the conditions (i) and (ii) in {Proposition~\ref{prop3:stringeigenvalues}} are not satisfied.

\begin{theorem} \label{thm6:SUtwo1-jumps}
Let $G=\SU(n+1)$ and $K=\op{S}(\U(n)\times\U(1))$.
For $l,m,s\in\Z$ such that $l,m\geq 0$ and $l+m\leq n$, let $\tau_{{l,m,s}}$ be the irreducible representation of $K$ with highest weight
\begin{equation}\label{eq6:SUmu_lms}
\mu_{l,m,s}:= \pr\left(\sum_{j=1}^l \varepsilon_j - \sum_{j=n+1-m}^n \varepsilon_j  +s\varepsilon_{n+1} \right). 
\end{equation}
Assume $s\neq 0$ when $l\neq m$, and $s\neq 2(n-2l)$ when $l=m<n/2$. 
Then, the representation-spectral converse is valid for $(G,K,\tau_{l,m,s})$.
\end{theorem}

\begin{proof}
Write $\mu_{l,m}= \sum_{i=1}^n b_i\varepsilon_i$ and $\beta = -1-\tfrac{l-m+s}{n+1}$; then $b_i=2+\beta$ for $1\leq i\leq l$, $b_i=1+\beta $ for $l+1\leq i\leq n-m$, $b_i=\beta$ for $n+1-m\leq i\leq n$, and $b_{n+1}=s+\beta$. 
From \eqref{eq6:P_tauSU(n+1)} it is a simple matter to describe $\PP_{\tau_{l,m,s}}$, though one has to be careful by {splitting conveniently into} cases. 
For instance, when $l$ and $m$ are both positive and $l+m<n$ (which is the generic case), one has that
$
\PP_{\tau_{l,m,s}} = \{\Lambda_0(2,1), \,\Lambda_0(2,0),\, \Lambda_0(1,1),\, \Lambda_0(1,0)\},
$
where 
\begin{align} \label{eq6:SU(n+1)Lambda_0(c_1,c_2)}
\Lambda_0(c_1,c_2)
	=&\, (2+\beta+\max(0,-r(c_1,c_2)))\varepsilon_1 
	+(2+\beta)\sum_{i=2}^{l} \varepsilon_i
	+(c_1+\beta)\varepsilon_{l+1} \\ &
	+(1+\beta) \sum_{i=l+2}^{n-m}\varepsilon_i 
	+(c_2+\beta) \varepsilon_{n+1-m}
	+\beta \sum_{i=n+2-m}^n \varepsilon_i \notag\\
	&+(\beta-\max(0,r(c_1,c_2)))\varepsilon_{n+1}
\notag
\end{align}
and $r(c_1,c_2)=c_1+c_2-s-2$. 
Note that $r(c_1,c_2)$ does not change sign if $s\neq0$. Furthermore, 
\begin{align} \label{eq6:SULambda_0(c1,c2)-Lambda_0(c1',c2')}
\Lambda_0(c_1,c_2)-\Lambda_0(c_1',c_2') 
	&= (\max(0,-r(c_1,c_2))- \max(0,-r(c_1',c_2'))) \varepsilon_1
	+(c_1-c_1')\varepsilon_{l+1} \\
	&+(c_2-c_2') \varepsilon_{n+1-m}
	-(\max(0,r(c_1,c_2)) - \max(0,r(c_1',c_2'))) \varepsilon_{n+1}.
	\notag
\end{align}

We now proceed to check {the conditions (i) and (ii) in Proposition~\ref{prop3:stringeigenvalues}} for every pair $\Lambda_0,\Lambda_0'$ in $\PP_{\tau_{l,m,s}}$. 
We will consider only the generic case, $l,m>0$ and $l+m<n$. The rest of the cases are simpler and left to the reader. 
Furthermore, the case $l+m=n$ is included in Theorem~\ref{thm6:SUonejumparbitrario} below. 

From \eqref{eq6:SULambda_0(c1,c2)-Lambda_0(c1',c2')}, it follows immediately that $\langle\dir, \Lambda_0(c_1,c_2)-\Lambda_0(c_1',c_2') \rangle = |r(c_1,c_2)|-|r(c_1',c_2')|$, thus 
$$
\frac{\langle\dir, \Lambda_0(c_1,c_2)-\Lambda_0(c_1',c_2')\rangle} {\langle\dir, \dir\rangle}=\pm \frac12 \notin \Z
$$ 
for any pair of coefficients $\{(c_1,c_2),(c_1',c_2')\}$ equal {to one of} $\{(2,1),(2,0)\}$, $\{(2,1),(1,1)\}$, $\{(2,0),(1,0)\}$, {or} $\{(1,1),(1,0)\}$.
Therefore, (i) and (ii) do not hold for those pairs. 
It remains to consider the cases of $\{(2,1),(1,0)\}$ and $\{(2,0),(1,1)\}$. 

One has that $\Lambda_0(2,0)-\Lambda_0(1,1)= \varepsilon_{l+1}-\varepsilon_{n+1-m}$, thus $\langle\dir, \Lambda_0(2,0)-\Lambda_0(1,1)\rangle = 0$ and 
\begin{align*}
\lambda(C,\pi_{\Lambda_0(2,0)}) - \lambda(C,\pi_{\Lambda_0(1,1)}) 
&= \langle \Lambda_0(2,0),\Lambda_0(2,0)+2\rho_G \rangle 
-\langle \Lambda_0(1,1),\Lambda_0(1,1)+2\rho_G \rangle \\
&= \langle \Lambda_0(2,0)- \Lambda_0(1,1),\Lambda_0(2,0)+\Lambda_0(1,1)+2\rho_G \rangle \\
&= 2(n+1-l-m)\neq0.
\end{align*} 
We conclude that (i)--(ii) do not hold for this pair.

We end the proof by analyzing the case of $\{\Lambda_0(2,1), \Lambda_0(1,0)\}$. 
We have that $\langle\dir, \Lambda_0(2,1)-\Lambda_0(1,0)\rangle = |1-s|-|-1-s|$, thus this quantity vanishes if and only if $s=0$. 
In this case (i.e.\ $s=0$), one can check that $\lambda(C,\pi_{\Lambda_0(2,1)}) - \lambda(C,\pi_{\Lambda_0(1,0)})= 2(l-m)$, so (i) {does not hold, if} $m\neq l$.  
However, when $s=0$ and $m=l$, one can check that $\pi_{\Lambda_0(2,1)}$ and $\pi_{\Lambda_0(1,0)}$ are conjugate to each other (see for instance \cite[VI.(5.1)]{BrockerDieck}), thus \eqref{eq3:dual} holds for this pair.
Suppose $s\neq0$. One has $\frac{\langle\dir, \Lambda_0(2,1)-\Lambda_0(1,0)\rangle} {\langle\dir, \dir\rangle}=1$ and {in \eqref{eq3:condrara} we find  $4(n-s+2)\neq 4(n-s+2-l+m)$,  since  $l\neq m$.} 
Hence, {(ii) does not hold} for $\Lambda_0(2,1)$ and $\Lambda_0(1,0)$, and the proof is complete. 
\end{proof}

Theorem~\ref{thm6:SUtwo1-jumps} already yields the validity of the representation-spectral converse for $(G,K,\tau)$ for infinitely many irreducible representations $\tau$ of $K$.  
The next result provides additional infinite choices of $\tau_\mu \in\widehat K$ with $\mu$ having one arbitrary jump among the first $n$ coefficients.

\begin{theorem}\label{thm6:SUonejumparbitrario}
	Let $G=\SU(n+1)$ and $K=\op{S}(\U(n)\times\U(1))$.
	For $t,s,l\in\Z$ such that $t\geq 0$, $1\leq l\leq n-1$, let $\tau_{t,l,s}'$ be the irreducible representation of $K$ with highest weight
	\begin{equation}
	\mu_{t,l,s}':= \pr\left(t\sum_{j=1}^l \varepsilon_j  +s\varepsilon_{n+1} \right). 
	\end{equation}
	Assume $s \leq 0$ or $s\geq t$, and $(t-n+2l-3s)/3\notin\Z\cap [1,t-1]$. 
	Then, the strong representation-spectral converse is valid for $(G,K,\tau_{t,l,s}')$.  
\end{theorem}

\begin{proof}
	We will omit many details in this proof since  the argument is very similar to  that in Theorem~\ref{thm6:SUtwo1-jumps}. 
	Write $\mu_{t,l,s}'= \sum_{i=1}^{n+1} b_i\varepsilon_i$ and $\beta = -\tfrac{tl+s}{n+1}$, then $b_i=t+\beta $ for $1\leq i\leq l$, $b_i=\beta $ for $l+1\leq i\leq n$, and $b_{n+1}=s+\beta$. 
	Then
	$
	\PP_{\tau_{t,l,s}'} = \{\Lambda_0(c):0\leq c\leq t\},
	$
	where 
	\begin{align} \label{eq6:SU(n+1)Lambda_0(c)}
	\Lambda_0(c)=&
	(t+\beta+\max(0,s-c))\varepsilon_1 
	+(t+\beta)\sum_{i=2}^{l} \varepsilon_i
	+(c+\beta)\varepsilon_{l+1} \\ &
	+\beta \sum_{i=p+2}^n \varepsilon_i 
	+(\beta-\max(0,c-s))\varepsilon_{n+1}.
	\notag
	\end{align}

We now check {conditions (i) and (ii) in Proposition~\ref{prop3:stringeigenvalues}} for every pair $\Lambda_0,\Lambda_0'$ in $\PP_{\tau_{t,l,s}}$. 
From \eqref{eq6:SU(n+1)Lambda_0(c)}, $\langle\dir, \Lambda_0(c)-\Lambda_0(c') \rangle = |c-s|-|c'-s|$,  which  vanishes if and only if $c=c'$.
In fact, since  by assumption $s\leq 0$ or $s\geq t$, one has that $c-s$ does not change sign provided that $0\leq c\leq t$.
It remains to show that {for  $c\neq c'$ such that $0\leq c,c'\leq t$, condition (ii) cannot hold for $\Lambda_0(c)$ and $\Lambda_0(c')$.} 
	
Fix $0\leq c'<c\leq t $. 
For simplicity, we assume that $s\leq 0$. 
The other case is completely  analogous. 
By \eqref{eq6:SU(n+1)Lambda_0(c)}, one obtains $\Lambda_0(c)-\Lambda_0(c') = (c-c')(\varepsilon_{l+1}- \varepsilon_{n+1})$.
	
We calculate
	\begin{align*}
	\langle \dir ,\Lambda_0-\Lambda_0'\rangle  &= (c-c'), &
	\langle \dir ,\Lambda_0+\Lambda_0'+2\rho_G\rangle
	&= (2n+2t+c+c'-2s), \\
	\langle \dir ,\dir \rangle &=2, &
	\langle \Lambda_0-\Lambda_0' ,\Lambda_0+\Lambda_0'+ 2\rho_G\rangle &=  2(c-c') (n-l+s+c+c').
	\end{align*}
	{This tells that} \eqref{eq3:condrara}  does not hold unless $2t = 3(c+c')+2n-4l+6s$, or equivalently $\Z\ni \frac{c+c'}{2}=\frac{t -n+2l-3s}{3}$, which contradicts our assumption. 
	This shows that (ii) is {not} satisfied for $\Lambda_0(c)$ and $\Lambda_0(c')$ and thus the proof is complete. 
\end{proof}

{We now give examples of disjoint strings having {infinitely many} coincidences of eigenvalues.
They} show that the assumptions in Theorem~\ref{thm6:SUonejumparbitrario} cannot  be substantially improved.

\begin{example}\label{ex6:SUunknowncases1}
{
	Let $\Lambda_0 =  (n-1)\varepsilon_1 + \varepsilon_2 -\sum_{i=3}^n \varepsilon_i -2\varepsilon_{n+1}$ and $\Lambda_0'=n\varepsilon_1 -\sum_{i=2}^n \varepsilon_i -\varepsilon_{n+1}$.
	One can easily check that $\langle\dir ,\Lambda_0-\Lambda_0'\rangle =0$ and $\langle \Lambda_0,\Lambda_0+2\rho_G\rangle = \langle \Lambda_0',\Lambda_0'+2\rho_G\rangle$, thus $\lambda(C,\pi_{k\dir+\Lambda_0}) = \lambda(C,\pi_{k\dir+\Lambda_0'})$ for all $k\geq0$ by Proposition~\ref{prop3:stringeigenvalues}. }

	Let $\mu= \pr(n\varepsilon_1+ \varepsilon_{n+1})$.
	In the notation in the proof of Theorem~\ref{thm6:SUonejumparbitrario}, we have that $t=n$, $p=1$, $s=1$, thus $\Lambda_0 := \Lambda_0(2)$ and $\Lambda_0':= \Lambda_0(0)$ are in $\PP_{\mu}$. 
Since \eqref{eq3:dual} is not possible, then we have infinitely many coincidences of eigenvalues.
\end{example}

\begin{example}\label{ex6:SUunknowncases2}
Let $\mu=\pr((n+1)\varepsilon_{1})$. 
In the notation of Theorem~\ref{thm6:SUonejumparbitrario}, $t=n+1$, $p=1$, $s=0$, thus $(t-n+2p-3s)/3=1$ contradicting the assumptions. 
In this opportunity, one has that (ii) in Proposition~\ref{prop3:stringeigenvalues} does hold for $\Lambda_0(2)$ and $ \Lambda_0(0)$. 
\end{example}

\subsubsection*{Application to $p$-form representations}
We now consider the representation $\tau_p$ of $K$ introduced in Definition~\ref{def2:tau_p}. 
{As is well known, $\tau_p$ is highly reducible and one has that} (see \cite[\S3]{IkedaTaniguchi78}) 
\begin{equation}\label{eq6:SUtau_p}
\tau_p \simeq \bigoplus_{l+m=p} \bigoplus_{j=0}^{\min(l,m)} \tau_{{l-j,m-j,m-l}},
\end{equation}
where $\tau_{{l,m,s}} \in \widehat K$ has highest weight $\mu_{l,m,s}$ defined in \eqref{eq6:SUmu_lms}.
Consequently, Theorem~\ref{thm6:SUtwo1-jumps} proves in particular that the representation-spectral converse is valid for $(G,K,\tau)$ for every irreducible constituent $\tau$ of $\tau_p$. {However, as we shall see below in Example~\ref{ex6:SUp-isocounterex}, the situation is very different for $\tau_p$, for any $p>2$.}

We next show a positive result for $\tau=\tau_p$ with $p=0,1$.

\begin{theorem}\label{thm6:SUp-iso}
The strong representation-spectral converse is valid for $(\SU(n+1),\op{S}(\U(n)\times\U(1)),\tau_{p})$ for $p=0,1$.  
\end{theorem}

\begin{proof}
The case $p=0$ is clear. We thus assume $p=1$. 
By \eqref{eq6:SUtau_p}, $\tau_1\simeq \tau_{1,0,-1}\oplus \tau_{0,1,1}$,
thus 
\begin{align*}
\PP_{\tau_1} &= \PP_{\tau_{1,0,-1}}\cup \PP_{\tau_{0,1,1}} 
=
\{\Lambda_0:=\varepsilon_1-\varepsilon_{n+1}, \; \Lambda_0':=\varepsilon_1+\varepsilon_2-2\varepsilon_{n+1},\;
\Lambda_0'':=2\varepsilon_1-\varepsilon_{n}-\varepsilon_{n+1}\}
\end{align*} 
by Lemma~\ref{lem6:hatG_tauSU(n+1)}. 
It is clear that $\pi_{\Lambda_0'}$ and $\pi_{\Lambda_0''}$ are dual to each other, and $\pi_{\Lambda_0}$ is self-dual. 

It only remains to check that the conditions (i)--(ii) {in Proposition~\ref{prop3:stringeigenvalues} do not occur} for  $\{\Lambda_0,\Lambda_0'\}$ and $\{\Lambda_0,\Lambda_0''\}$. This follows from the identities $$\langle \dir,\Lambda_0-\Lambda_0'\rangle / \langle\dir,\dir\rangle =\langle \dir,\Lambda_0-\Lambda_0''\rangle / \langle\dir,\dir\rangle= 1/2,$$
and the proof is complete.
\end{proof}

We now consider the representations $\tau_p$ for $p\ge 2$.

\begin{example}\label{ex6:SUp-isocounterex}
{
Let $\Lambda_0 := \varepsilon_1+\varepsilon_2 -\varepsilon_n-\varepsilon_{n+1}$, $\Lambda_0' := \varepsilon_1+\varepsilon_2 +\varepsilon_3-3\varepsilon_{n+1}$, and $\Lambda_0'' := 3\varepsilon_1-\varepsilon_{n-1} -\varepsilon_n-\varepsilon_{n+1}$. 
Since  $\Lambda_0-\Lambda_0'= -\varepsilon_3-\varepsilon_n+2\varepsilon_{n+1}$ and $\Lambda_0-\Lambda_0''= -2\varepsilon_{1}+\varepsilon_{2}-\varepsilon_{n-1}$,
an easy computation shows that 
\begin{align}
\frac{\langle \dir ,\Lambda_0-\Lambda_0'\rangle}{\langle \dir ,\dir \rangle}  = 
\frac{ \langle \Lambda_0-\Lambda_0' ,\Lambda_0+\Lambda_0'+2\rho_G\rangle} {\langle \dir ,\Lambda_0+\Lambda_0'+2\rho_G\rangle} = -1,
\end{align}
and the same {identities} hold when replacing $\Lambda_0'$ by $\Lambda_0''$. 
Hence, Proposition~\ref{prop3:stringeigenvalues} (ii) implies that 
$
\lambda(C,\pi_{\Lambda_0 + k \dir}) = \lambda(C,\pi_{\Lambda'_0 + (k+1)\dir})
 = \lambda(C,\pi_{\Lambda''_0 + (k+1)\dir})
$
for all $k\geq0$.

From Lemma~\ref{lem6:SUbranchinglaw}, we can verify the following facts:
\begin{itemize}
	\item $\tau_{2,1,-1}, \tau_{1,1,0},\tau_{1,2,1}$ {occur} in the decomposition of $\pi_{\Lambda_0}|_K$, thus $\Lambda_0 \in \PP_{\tau_{2,1,-1}}\cap \PP_{\tau_{1,1,0}}\cap \PP_{\tau_{1,2,1}}$;
	
	\item $\tau_{3,0,-3}, \tau_{2,1,-1}, \tau_{2,0,-2}$ {occur} in $\pi_{\Lambda_0'}|_K$, thus $\Lambda_0' \in \PP_{\tau_{3,0,-3}}\cap \PP_{\tau_{2,1,-1}}\cap \PP_{\tau_{2,0,-2}}$;
	
	\item $\tau_{0,3,3}, \tau_{1,2,1}, \tau_{0,2,2}$ {occur} in $\pi_{\Lambda_0''}|_K$, thus $\Lambda_0'' \in \PP_{\tau_{0,3,3}}\cap \PP_{\tau_{1,2,1}}\cap \PP_{\tau_{0,2,2}}$.
\end{itemize}
By \eqref{eq6:SUtau_p}, $\tau_{1,1,0}$, $\tau_{2,0,-2}$ and $\tau_{0,2,2}$ are irreducible constituents of $\tau_p$ for all $p\geq 2$, $p$ even. 
Similarly, $\tau_{2,1,-1}$, $\tau_{1,2,1}$, $\tau_{3,0,-3}$, and $\tau_{0,3,3}$ are irreducible constituents of $\tau_p$ for all $p\geq3$, $p$ odd. 
Consequently, $\Lambda_0$, $\Lambda_0'$ and $\Lambda_0''$ belong to $\PP_{\tau_p}$ for all $p\geq2$.

	Furthermore, one can easily check that $\pi_{\Lambda_0'}$ and $\pi_{\Lambda_0''}$ are dual to each other, and $\pi_{\Lambda_0}$ is self-dual, thus \eqref{eq3:dual} does not hold for the pairs $\{\Lambda_0,\Lambda_0'\}$ and $\{\Lambda_0,\Lambda_0''\}$.
	From the discussion above, we conclude that the hypotheses in Theorem~\ref{thm3:tauiso=>tauequiv} are not satisfied for $\tau_p$ for {each} $p\geq2$. 
}
\end{example}

\section{Quaternionic projective spaces}
\label{sec:Sp}
Now we consider the quaternionic projective space $P^n(\Hy)$, realized as $G/K$ with
\begin{align*}
G&=\Sp(n+1) = \left\{g\in \SU(2n+2):  g^t\left(\begin{smallmatrix} 0&I\\ -I&0 \end{smallmatrix}\right)g=\left(\begin{smallmatrix} 0&I\\ -I&0 \end{smallmatrix}\right)  \right\}, \\
K&=\Sp(n)\times \Sp(1) = \left\{ \left(\begin{matrix} g_1&\\ &g_2 \end{matrix}\right): g_1\in\Sp(n),\, g_2\in\Sp(1) \right\}.
\end{align*}
We consider the metric induced by  the inner product  $\langle X,Y\rangle = -\tfrac12\tr(XY)$. 

We fix the maximal torus in $G$ given by 
\begin{equation*}
T= \{\diag(e^{i\theta_1},\dots ,e^{i\theta_{n+1}}, e^{-i\theta_1}, \dots,e^{-i\theta_{n+1}}): \theta_j\in\R\}. 
\end{equation*}
Then, any element in $\mathfrak t$ (resp.\ $\mathfrak t_\C$) has the form 
$$
X=\diag(i\theta_1, \dots,i\theta_{n+1},-i\theta_1,\dots, -i\theta_{n+1})
$$
with $\theta_j\in \R$ (resp.\ $\theta_j\in\C$) for all $j$. 
We set $\varepsilon_j\in\mathfrak t_\C^*$ given by $\varepsilon_j(X)=i\theta_j$ where $X$ is the element in $\mathfrak t_\C$ given above. 
It turns out that $\Phi(\mathfrak g_\C,\mathfrak t_\C)= \{\pm (\varepsilon_i-\varepsilon_j): 1\leq i<j\leq n+1\}\cup \{ \pm 2\varepsilon_i: 1\leq i\leq n+1\}$ and 
$
\PP(G)= \bigoplus_{i=1}^{n+1} \Z\varepsilon_i .
$ 
Furthermore, $\langle \varepsilon_i,\varepsilon_j\rangle =-\delta_{i,j}$ for all $1\leq i,j\leq n+1$.
 In  the standard order, the set of simple roots is $\{\varepsilon_1-\varepsilon_{2},\dots, \varepsilon_n-\varepsilon_{n+1}, 2\varepsilon_{n+1}\}$, $\Phi^+(\mathfrak g_\C,\mathfrak t_\C)= \{\varepsilon_i-\varepsilon_j: 1\leq i<j\leq n+1\}\cup \{2\varepsilon_i:1\leq i\leq n+1 \}$, $\rho_G=\sum_{i=1}^{n+1} (n+4-2i)\varepsilon_{i}$, and furthermore 
$$
\PP^{+}(G) = \{\textstyle \sum_{j=1}^{n+1} a_j\varepsilon_j\in \PP(G): a_1\geq\dots\geq a_{n+1}\geq0\}. 
$$

The group $T$ is also a maximal torus for $K$. 
Hence, $\Phi(\mathfrak k_\C,\mathfrak t_\C)= \{\pm (\varepsilon_i-\varepsilon_j): 1\leq i<j\leq n\}\cup \{ \pm 2\varepsilon_{n+1}\}$, 
$
\PP(K)=\PP(G),
$ 
$\langle \varepsilon_i,\varepsilon_j\rangle =\delta_{i,j}$ for all $1\leq i,j\leq n+1$. 
{With the induced order, $K$} has simple roots $\{\varepsilon_1-\varepsilon_{2},\dots, \varepsilon_{n-1}-\varepsilon_{n}, 2\varepsilon_{n}, 2\varepsilon_{n+1}\} $, $\Phi^+(\mathfrak{k}_\C,\mathfrak t_\C)= \{\varepsilon_i-\varepsilon_j: 1\leq i<j\leq n\}\cup \{2\varepsilon_1\}$, and 
$$
\PP^{+}(K) = \{\textstyle \sum_{j=1}^{n+1} b_j\varepsilon_j\in \PP(K): b_1\geq\dots\geq b_{n}\geq0, \, b_{n+1}\geq0\}. 
$$

The branching law in this case was proved by Lepowsky~\cite{Lepowsky71}  and  presents deeper difficulties than in the orthogonal and unitary cases.
We will use the following alternative statement \cite[Thm.~1.3]{Tsukamoto81} by Tsukamoto.
Other statements can be found for instance in \cite[Thms.~4.1--4.2]{Camporesi05Pacific}. 

\begin{lemma}\label{lem7:Spbranchinglaw}
Let $G=\Sp(n+1)$ and $K=\Sp(n)\times\Sp(1)$. 
If $\Lambda=\sum_{j=1}^{n+1} a_j\varepsilon_j\in \PP^{+}(G)$ and $\mu=\sum_{j=1}^{n+1} b_j\varepsilon_j\in \PP^{+}(K)$, then $\tau_\mu$ does not occur in $\pi_{\Lambda}|_K$ unless 
\begin{align}\label{eq7:Spentrelazamiento}
	a_i\geq b_{i}\geq a_{i+2}\quad\text{for all }1\leq i\leq n, 
\end{align} 
where $a_{n+2}=0$.
Furthermore, when \eqref{eq7:Spentrelazamiento} holds, $\dim \Hom_K(\tau_\mu,\pi_\Lambda)$ equals the coefficient of $x^{b_{n+1}+1}$ in the power series expansion in $x$ of 
\begin{equation}
\frac{1}{(x-x^{-1})^{n}} \prod_{i=1}^{n+1} (x^{\delta_i+1}-x^{-\delta_i-1}),
\end{equation}
where 
\begin{align}\label{eq7:Sp-condnec}
\left\{
\begin{array}{r@{\,}c@{\,}l}
\delta_1&=&a_1-\max(a_2,b_1),\\[1mm] 		\delta_i&=&\min(a_i,b_{i-1})-\max(a_{i+1},b_i)\quad \text{for }2\leq i\leq n, \\[1mm]
\delta_{n+1} &=& \min(a_{n+1},b_n). 
\end{array}
\right.
\end{align}
\end{lemma}

It is important to note that the doubly interlacing  condition  \eqref{eq7:Spentrelazamiento} is a necessary but not a  sufficient condition {to have} $\dim \Hom_K(\tau_\mu,\pi_\Lambda)>0$.

It is well known that $\widehat G_{1_K}=\String{\dir,0}=\{\pi_{k\dir}:k\in\N_0\}$, where $\dir=\varepsilon_1+\varepsilon_2$. 
Camporesi in \cite[\S{}4]{Camporesi05Pacific} gave, for some particular choices of $\tau\in\widehat K$, an explicit parametrization of $\PP_\tau$ satisfying that $\widehat G_\tau=\bigcup_{\Lambda_0\in\PP_\tau} \String{\dir,\Lambda_0}$ (see also \cite[\S{}5]{HeckmanPruijssen16}). 
The next result provides infinitely many choices of $\tau$ such that $\widehat G_\tau$ is written as a finite disjoint union of strings of representations and furthermore, the strong representation-spectral converse holds for $(G,K,\tau)$. 

For non-negative integers $m,s$ satisfying $m\leq n-1$, let $\tau_{m,s}$ denote the irreducible representation of $K$ with highest weight 
\begin{equation}
\mu_{m,s}:= \sum_{i=1}^{m} \varepsilon_i + s\varepsilon_{n+1} \in \PP^{+}(K). 
\end{equation}
Lemma~\ref{lem7:Spbranchinglaw} immediately implies the following description of $\widehat G_{\tau_{m,s}}$.

\begin{lemma} \label{lem7:hatG_tauSp}
For $m,s\in\N_0$ and $m\leq n-1$, we have that
\begin{equation}
\widehat G_{\tau_{m,s}} 
= \bigcup_{\Lambda_0\in \PP_{\tau_{m,s}}}\{\pi_{k\dir +\Lambda_0}:k\geq0\} 
= \bigcup_{\Lambda_0\in \PP_{\tau_{m,s}}} \String{\dir,\Lambda_0},
\end{equation}
where $\dir=\varepsilon_1+\varepsilon_2$, $\PP_{\tau_{0,s}} = \{s\varepsilon_1\}$, and for $m\geq1$, 
\begin{equation}\label{eq7:SpPP_tau_ms}
\PP_{\tau_{m,s}} = \left\{ 
\begin{array}{ll}
\Lambda_0':=(s+2)\varepsilon_1+ \sum\limits_{i=2}^{m+1}\varepsilon_i, &
\Lambda_0'':=s\varepsilon_1+ \sum\limits_{i=2}^{m+1}\varepsilon_i   \quad (\text{if $s\geq1$}), \\
\Lambda_0''':=(s+1)\varepsilon_1+ \sum\limits_{i=2}^{m}\varepsilon_i, &
\Lambda_0'''':=(s+1)\varepsilon_1+ \sum\limits_{i=2}^{m+2}\varepsilon_i
\end{array}
\right\}.
\end{equation}
\end{lemma}

\begin{remark}
The case when $s=1$ in Lemma~\ref{lem7:hatG_tauSp} is a particular case of \cite[Thm.~5.2]{HeckmanPruijssen16}.
\end{remark}

\begin{theorem} \label{thm7:Spone1-jump}
Let $G=\Sp(n+1)$, $K=\Sp(n)\times\Sp(1)$, and $m,s\in\N_0$, with  $m\leq n-1$.
Then, the strong representation-spectral converse is valid for  $(G,K,\tau_{m,s})$. 
\end{theorem}

\begin{proof}
The case $\mu_{0,s} = s\varepsilon_{n+1}$ is obvious  since $\widehat G_{\tau_{0,s}} = \String{\dir, s\varepsilon_1}$.
We thus assume $m > 0$. 
We consider  conditions (i) and (ii) in Proposition~\ref{prop3:stringeigenvalues} for each pair of elements in $\PP_{\tau_{m,s}}$ (see \eqref{eq7:SpPP_tau_ms}).
One has that $\langle \dir,\Lambda_0'-\Lambda_0'''\rangle/ \langle\dir,\dir\rangle =1/2\notin\Z$, thus (ii) is not satisfied for $\Lambda_0',\Lambda_0'''$. 
The same argument applies to all pairs except for $\{\Lambda_0',\Lambda_0''\}$ and $\{\Lambda_0''',\Lambda_0''''\}$. 
{
In the first case, one has that 
$$
\langle \dir, \Lambda_0'-\Lambda_0'' \rangle / \langle\dir,\dir\rangle=1\neq \frac{4n+4s+4}{4n+2s+2} = \frac{\langle \Lambda_0' -\Lambda''_0, \Lambda_0'+\Lambda''_0 +2\rho_G\rangle} {\langle \dir ,\Lambda_0'+\Lambda''_0+2\rho_G\rangle},
$$
thus (ii) does not hold. }
In the case of $\{\Lambda_0''',\Lambda_0''''\}$, it is a simple matter to check that $\langle \dir,\Lambda_0'''- \Lambda_0''''\rangle=0$ and $\lambda(C,\pi_{\Lambda_0'''}) - \lambda(C,\pi_{\Lambda_0''''}) = (4n-4m+2)\neq 0$. 
This completes the proof, in light of Theorem~\ref{thm3:tauiso=>tauequiv}. 	
\end{proof}

\subsubsection*{Applications to $p$-form representations}
We now consider the representations $\tau_p$ associated to the $p$-form spectrum. 
Tsukamoto in \cite[page 421]{Tsukamoto81} explained an algorithm to decompose $\tau_p$ as a sum of irreducible representations of $K$.
He also gave the following explicit expressions for the first five cases: 
\begin{align}\label{eq7:Sptau_p}
\tau_0&\simeq 1_K, \qquad \tau_1\simeq \tau_{\varepsilon_{1}+\varepsilon_{n+1}}, \qquad 
\tau_2\simeq \tau_{2\varepsilon_1} \oplus \tau_{2\varepsilon_{n+1}}\oplus \tau_{\varepsilon_{1}+\varepsilon_{2}+2\varepsilon_{n+1}}, \\ \notag
\tau_3 & \simeq \tau_{\varepsilon_{1}+\varepsilon_{n+1}} \oplus \tau_{2\varepsilon_{1}+\varepsilon_{2}+\varepsilon_{n+1}} \oplus \tau_{\varepsilon_{1}+3\varepsilon_{n+1}} \oplus \tau_{\varepsilon_{1}+\varepsilon_{2}+\varepsilon_{3}+3\varepsilon_{n+1}} ,\\ \notag
\tau_4 & \simeq 
\tau_0 \oplus 
\tau_{\varepsilon_{1}+\varepsilon_{2}} \oplus 
\tau_{2\varepsilon_{1}+2\varepsilon_{2}} \oplus
\tau_{\varepsilon_{1}+\varepsilon_{2}+2\varepsilon_{n+1}} \oplus
\tau_{2\varepsilon_{1}+2\varepsilon_{n+1}} \oplus
\tau_{2\varepsilon_{1}+\varepsilon_{2}+\varepsilon_{3}+2\varepsilon_{n+1}} \oplus \\ \notag &\qquad 
\tau_{4\varepsilon_{n+1}} \oplus 
\tau_{\varepsilon_{1}+\varepsilon_{2}+4\varepsilon_{n+1}} \oplus
\tau_{\varepsilon_{1}+\varepsilon_{2}+\varepsilon_{3}+\varepsilon_{4}+4\varepsilon_{n+1}}.
\end{align}
However, to our best knowledge, there is no known explicit expression valid for every $p$.
The situation is {even} worse for quaternionic Grassmann spaces (cf.\ \cite[\S{}3]{Chami12}). 

We now show the strong representation-spectral converse for $(G,K,\tau_p)$ for $p=0,1$. 
It remains open to us for $p\ge 2$.

\begin{theorem}\label{thm7:Spp-iso}
The strong representation-spectral converse is valid for $(\Sp(n+1), \Sp(n)\times\Sp(1) ,\tau_{p})$ for $p=0,1$ and also for $(\Sp(n+1), \Sp(n)\times\Sp(1) ,\tau)$ where $\tau$ is any irreducible constituent of $\tau_2$. 
\end{theorem}

\begin{proof}
We will apply Theorem~\ref{thm3:tauiso=>tauequiv}. 
We let $\dir=\varepsilon_{1}+\varepsilon_{2}$. 
The case $p=0$ follows immediately since $\widehat G_{\tau_0} = \String{\dir,0}$. 
If $p=1$, then $\tau_1$ is irreducible with highest weight $\varepsilon_1+\varepsilon_{n+1}$, thus this case was already shown in Theorem~\ref{thm7:Spone1-jump}. 

The irreducible constituents of $\tau_2$ are $\tau_{2\varepsilon_1}$, $\tau_{2\varepsilon_{n+1}}$, and $\tau_{\varepsilon_{1}+\varepsilon_{2}+2\varepsilon_{n+1}}$ by \eqref{eq7:Sptau_p}. 
The last two cases follow by Theorem~\ref{thm7:Spone1-jump}. 
One can check by using Lemma~\ref{lem7:Spbranchinglaw} that $\widehat G_{\tau_{2\varepsilon_1}} = \bigcup_{\Lambda_0\in \PP_{\tau_{2\varepsilon_1}}} \String{\dir,\Lambda_0}$, where
$\PP_{\tau_{2\varepsilon_1}} = \{2\varepsilon_1,\, 2\varepsilon_1+\varepsilon_2+\varepsilon_3,\, 2\varepsilon_1+2\varepsilon_2+2\varepsilon_3 \}$.
Then, it is a simple matter to check that conditions (i) and (ii) in Proposition~\ref{prop3:stringeigenvalues} are not satisfied for any pair of elements in $\PP_{\tau_{2\varepsilon_1}}$ and hence Theorem~\ref{thm3:tauiso=>tauequiv} applies. 
\end{proof}

\begin{example}\label{ex7:Spunknowncases}
Let $\Lambda_0= 3\varepsilon_{1}+ \varepsilon_{2}+ \varepsilon_{3}+\varepsilon_{4}$ and $\Lambda_0'= 2\varepsilon_{1}+2\varepsilon_{2}+2\varepsilon_{3}$. 
Then $\Lambda_0-\Lambda_0'= \varepsilon_{1}-\varepsilon_{2}- \varepsilon_{3}+\varepsilon_{4}$, thus $\langle \dir, \Lambda_0-\Lambda_0'\rangle =0$ and 
\begin{align}
\lambda(C,\pi_{\Lambda_0}) - \lambda(C,\pi_{\Lambda_0'}) 
	&= \langle \Lambda_0-\Lambda_0', \Lambda_0+\Lambda_0'+2\rho_G\rangle\\
	&= (5+2n+2)-(3+2n)-(3+2n-2)+(1+2n-4)=0. \notag
\end{align}
Hence, $\lambda(C,\pi_{k\dir+\Lambda_0}) = \lambda(C,\pi_{k\dir+\Lambda_0'})$ for all $k\geq0$ by Proposition~\ref{prop3:stringeigenvalues}.

Let $\mu = 2\varepsilon_{1}+\varepsilon_{2}+\varepsilon_{n+1}$.
One can check that $\widehat G_{\tau_{\mu}} = \bigcup_{\Lambda_0\in \PP_{\tau_{\mu}}} \String{\dir,\Lambda_0}$, where  
\begin{align*}
\PP_{\tau_{\mu}} = \left\{
\begin{array}{lll}
3\varepsilon_{1}+\varepsilon_{2}, &
4\varepsilon_{1}+\varepsilon_{2}+\varepsilon_{3}, &
3\varepsilon_{1}+\varepsilon_{2}+\varepsilon_{3}+\varepsilon_{4}, \\
2\varepsilon_{1}+2\varepsilon_{2}, &
2\varepsilon_{1}+2\varepsilon_{2}+2\varepsilon_{3}, &
2\varepsilon_{1}+2\varepsilon_{2}+\varepsilon_{3}+\varepsilon_{4}, \\
2\varepsilon_{1}+\varepsilon_{2}+\varepsilon_{3}, &
4\varepsilon_{1}+2\varepsilon_{2}+2\varepsilon_{3}, &
3\varepsilon_{1}+2\varepsilon_{2}+2\varepsilon_{3}+\varepsilon_{4}
\end{array}
\right\}.
\end{align*}
Since $\Lambda_0,\Lambda_0'\in \PP_{\tau_{\mu}}$, we cannot ensure that the strong representation-spectral converse is valid for $(\Sp(n+1),\Sp(n)\times\Sp(1), \tau_\mu)$. 
Moreover, since $\tau_{\mu}$ is an irreducible constituent of $\tau_3$ by \eqref{eq7:Sptau_p}, then $\PP_{\tau_{\mu}}\subset \PP_{\tau_3}$ (see Remark~\ref{rem2:tau-reducible}), it follows that  $\tau_3$ does not satisfy the assumptions in Theorem~\ref{thm3:tauiso=>tauequiv} either. 

Furthermore, the strings $\String{\dir,\Lambda_0}$ and $\String{\dir,\Lambda_0'}$ belong to $\widehat G_{\tau_2}$ and consequently we cannot ensure the validity of the strong representation-spectral converse for $(G,K,\tau_2)$. 
In fact, we have seen in the proof of Theorem~\ref{thm7:Spp-iso} that $\Lambda_0\in\PP_{\tau_{2\varepsilon_1}}$.
Also, $\Lambda_0 \in \PP_{\tau_{\varepsilon_1+\varepsilon_2+2\varepsilon_{n+1}}}$ by Lemma~\ref{lem7:hatG_tauSp}. 
\end{example}

\section{Cayley plane}
\label{sec:F4}
As a final case, we will look at the Cayley plane (or octonion projective plane) $P^2(\mathbb{O})$, that corresponds to the pair $(G,K):=(\op{F}_4, \Spin(9))$. 
Here, $\op{F}_4$ denotes the simply connected compact Lie group with Lie algebra $\mathfrak f_4$, the compact Lie algebra whose complexification is of exceptional type $F_4$.

An explicit {general} branching rule for this pair is not available.
Particular cases were studied by Lepowsky~\cite{Lepowsky71} and by Mashimo~\cite{Mashimo97,Mashimo06}.
Heckman and van Pruijssen~\cite{HeckmanPruijssen16} described $\widehat G_\tau$ as a finite and disjoint union of strings of representations for every $\tau\in\widehat K$ such that $(G,K,\tau)$ is multiplicity free, i.e.\ $\dim \Hom_K(\pi,\tau)\leq 1$ for all $\pi\in\widehat G$. 
Furthermore, Mashimo~\cite{Mashimo97,Mashimo06} covered all the cases necessary to compute the spectrum on $p$-forms of $P^2(\mathbb{O})$ for every $p$. 
These results will be sufficient to exhibit an infinite set of $\tau \in \widehat K$ such that the  strong representation-spectral converse is valid for $(G,K,\tau)$.

To state the results, we introduce the notation of the positive root systems corresponding to $\op{F_4}$ and $\Spin(9)$, as chosen in \cite{HeckmanPruijssen16} (which coincides with \cite{Knapp-book-beyond}). 
We pick a common maximal torus $T$ in $\op{F}_4$ and $\Spin(9)$ with a basis $\{\varepsilon_1,\dots,\varepsilon_4\}$ of $\mathfrak t_\C^*$ and we use the $\Ad(G)$-invariant inner product $\langle\cdot,\cdot\rangle$ on $\mathfrak g$ such that the Hermitian extension of $\langle\cdot, \cdot\rangle|_{\mathfrak t}$ to $\mathfrak t_\C^*$ makes  $\{\varepsilon_1,\dots,\varepsilon_4\}$ orthonormal. 

We pick the positive system of roots such that
\begin{align*}
\Phi^+(\mathfrak k_\C,\mathfrak t_\C) 
&= \{\varepsilon_i \pm \varepsilon_j: 1\leq i<j\leq 4\} \cup \{\varepsilon_i: 1\leq i\leq 4\},\\
\Phi^+(\mathfrak g_\C,\mathfrak t_\C) 
&=\Phi^+(\mathfrak k_\C,\mathfrak t_\C) \cup \{\tfrac12 (\varepsilon_1\pm\varepsilon_2  \pm\varepsilon_3  \pm\varepsilon_4)\}.
\end{align*}
One can check that $\rho_G =\tfrac12(11\varepsilon_1+ 5\varepsilon_2+3\varepsilon_3+\varepsilon_4)$ and 
$$
\PP(G)=\PP(K)= \left\{ a_1\varepsilon_1 +a_2\varepsilon_2 + a_3\varepsilon_3 + a_4\varepsilon_4 : 2a_1,a_1-a_2,a_1-a_3,a_1-a_4\in\Z\right\}
$$ 
({equivalently},  $\sum_{i=1}^4 a_i\varepsilon_i\in \PP(G)$ if and only if the coefficients {$a_i$} are all integers or all half integers).
The fundamental weights for $\Phi(\mathfrak g_\C,\mathfrak t_\C) $ are 
\begin{align*}
\omega_1&:=\varepsilon_{1},& \omega_2&:=\tfrac12(3\varepsilon_1+ \varepsilon_2+ \varepsilon_3+ \varepsilon_4),&
\omega_3&:=2\varepsilon_{1}+\varepsilon_{2}+ \varepsilon_{3}, &
\omega_4&:=\varepsilon_{1}+\varepsilon_{2},
\end{align*}
and {those} for $\Phi(\mathfrak k_\C,\mathfrak t_\C) $ are
\begin{align*}
\upsilon_1&:=\varepsilon_{1},& \upsilon_2&:=\varepsilon_{1}+\varepsilon_{2},& \upsilon_3&:=\varepsilon_1+ \varepsilon_2+ \varepsilon_3,& 
\upsilon_4&:= \tfrac12(\varepsilon_1+ \varepsilon_2+ \varepsilon_3+ \varepsilon_4).
\end{align*}

The next result is proved in \cite[Prop.~6.3 and Thm.~6.4]{HeckmanPruijssen16}.

\begin{lemma}\label{lem8:hatG_tau}
Let $\tau \in\widehat K$ with highest weight $\mu=b_1\upsilon_1+b_2\upsilon_2$
with $b_1,b_2\in\N_0$. 
Then, 
\begin{equation}
\widehat G_\tau = \bigcup_{\Lambda_0\in \PP_\tau} \String{\dir_1,\Lambda_0} = \{\pi_{\Lambda_0+k\dir_1}:k\in\N_0\},
\end{equation}
where $\PP_\tau$ is given by 
\begin{align}
&\left\{
\sum_{i=1}^4a_i\omega_i:
\begin{array}{l}
	a_2+a_3+a_4\leq b_1+b_2,\\
	a_3+a_4\leq b_2\leq a_2+a_3+a_4,\\
	a_1=b_1+b_2-a_2-a_3-a_4
\end{array}
\right\}.
\end{align}
\end{lemma}

We are now in a position to prove the main result {in} this section.
The complexity of checking the conditions in Theorem~\ref{thm3:tauiso=>tauequiv} is high, for an arbitrary element $\tau_\mu$ as in Lemma~\ref{lem8:hatG_tau}, so we restrict our attention to those with $\mu\in \N_0\upsilon_1$.

\begin{theorem} \label{thm8:F4-infinity}
Let $G=\op{F}_4$, $K=\Spin(9)$.
Then, the representation-spectral converse is  valid for $(G,K,\tau_{b\upsilon_1})$, for any $b\in\N_0$. 
\end{theorem}

\begin{proof}
Throughout the proof, we fix $b\in\N_0$ and abbreviate $\tau=\tau_{b\upsilon_1}$. 
Lemma~\ref{lem8:hatG_tau} gives $\widehat G_\tau =\bigcup_{\Lambda_0\in \PP_\tau} \String{\dir_1,\Lambda_0}$ with
\begin{equation}
\PP_\tau=\{(b-a)\omega_1+ a\omega_2:0\leq a\leq b\}. 
\end{equation}
{It remains to check that conditions (i) and (ii) in Proposition~\ref{prop3:stringeigenvalues} do not hold, for each pair of elements in $\PP_{\tau}$.}

Let $\Lambda_0=(b-a)\omega_1+ a\omega_2, \Lambda_0'=(b-a')\omega_1+ a'\omega_2 \in\PP_\tau$ for some integers $0\leq a'<a\leq b$.  
Since $\Lambda_0-\Lambda_0'= (a-a')(\omega_2-\omega_1) = \frac{a-a'}{2}(\varepsilon_{1}+ \varepsilon_{2}+ \varepsilon_{3}+ \varepsilon_{4})$.
It follows that $\langle \dir_1,\Lambda_0-\Lambda_0'\rangle = \frac{a-a'}{2}\neq0$, therefore (i) does not hold.
Furthermore, it is a simple matter to check that
\begin{align*}
\langle \dir_1 ,\Lambda_0+\Lambda_0'+2\rho_G\rangle
&= (11+2b+\tfrac{a+a'}{2}),
\\
\langle \Lambda_0-\Lambda_0', \Lambda_0+\Lambda_0' +2\rho_G\rangle
&= \tfrac{a-a'}{2} (20+2b+2(a+a')), 
\end{align*}
which immediately gives that \eqref{eq3:condrara} does not hold, so also (ii) does not.
\end{proof}

\subsubsection*{Applications to $p$-form representations}
We now consider the representation-spectral converse on $p$-forms of $P^2(\mathbb O)$. 
We will assume that $0\leq p\leq 8$ without loosing generality since $\tau_p$ and $\tau_{16-p}$ are equivalent. Mashimo listed the irreducible representations $\tau_\mu$ of $K$ that are constituents of $\tau_p$ (see \cite[Table 1]{Mashimo97}), together with the representations $\pi_\Lambda$ in $\widehat G_{\tau_\mu}$ with  
multiplicity $[\tau_\mu: \pi_\Lambda|_K]\ne 0$ (see \cite[Tables 1--28]{Mashimo06}).
Mashimo concludes the calculation of every $p$-spectra of $P^2(\mathbb O)$ by giving the Casimir eigenvalues corresponding to the irreducible representations of $G$ such that the restriction to $K$ contains an 
irreducible representation of $K$ occurring in $\tau_p$ (see \cite[Table 2]{Mashimo97} for $0\leq p\leq 5$ and \cite[Tables 29--31]{Mashimo06} for $6\leq p\leq 8$).

In Mashimo's positive root system convention, the decomposition of the representation $\tau_p$ (see Definition~\ref{def2:tau_p}) is  as follows (see \cite[Subsect.~2.3]{Mashimo97}):
\begin{align}\label{eq8:F4tau_p}
\tau_0&\simeq 
	1_K, 
\qquad 
\tau_1\simeq 
	\tau_{\upsilon_4}, 
\qquad 
\tau_2\simeq 
	\tau_{\upsilon_2} \oplus \tau_{\upsilon_3}, 
\qquad
\tau_3  \simeq 
	\tau_{\upsilon_2+\upsilon_4} \oplus
	\tau_{\upsilon_1+\upsilon_4},
\\ \notag
\tau_4 & \simeq 
	\tau_{2\upsilon_4} \oplus 
	\tau_{2\upsilon_1} \oplus 
	\tau_{\upsilon_1+\upsilon_2} \oplus
	\tau_{\upsilon_1+2\upsilon_4} \oplus
	\tau_{2\upsilon_2}  ,
\\ \notag
\tau_5 & \simeq 
	\tau_{\upsilon_1+\upsilon_4} \oplus 
	\tau_{\upsilon_2+\upsilon_4} \oplus 
	\tau_{3\upsilon_4} \oplus
	\tau_{2\upsilon_1+\upsilon_4} \oplus
	\tau_{\upsilon_1+\upsilon_2+\upsilon_4} 
,\\ \notag
\tau_6 & \simeq 
	\tau_{\upsilon_2} \oplus 
	\tau_{\upsilon_3} \oplus 
	\tau_{\upsilon_1+\upsilon_2}\oplus
	\tau_{\upsilon_1+\upsilon_3} \oplus
	\tau_{\upsilon_1+2\upsilon_4} \oplus
	\tau_{\upsilon_2+2\upsilon_4} \oplus
	\tau_{2\upsilon_1+\upsilon_2} \oplus
	\tau_{2\upsilon_1+\upsilon_3} 
,\\ \notag
\tau_7 & \simeq 
	\tau_{\upsilon_4} \oplus 
	\tau_{\upsilon_1+\upsilon_4} \oplus 
	\tau_{\upsilon_2+\upsilon_4}\oplus
	\tau_{\upsilon_3+\upsilon_4} \oplus
	\tau_{2\upsilon_1+\upsilon_4} \oplus
	\tau_{\upsilon_1+\upsilon_2+\upsilon_4} \oplus
	\tau_{\upsilon_1+\upsilon_3+\upsilon_4} \oplus
	\tau_{3\upsilon_1+\upsilon_4} 
,\\ \notag
\tau_8 & \simeq 
	\tau_{0} \oplus 
	\tau_{\upsilon_1} \oplus
	\tau_{\upsilon_3} \oplus
	\tau_{2\upsilon_4} \oplus
	\tau_{2\upsilon_1} \oplus
	\tau_{\upsilon_1+\upsilon_3} \oplus 	
	\tau_{\upsilon_1+2\upsilon_4} \oplus 
	\tau_{2\upsilon_2} \oplus
	\tau_{\upsilon_2+\upsilon_3} 
\\ \notag &\quad  \oplus 
	\tau_{2\upsilon_3} \oplus
	\tau_{3\upsilon_1} \oplus
	\tau_{2\upsilon_1+\upsilon_3} \oplus
	\tau_{2\upsilon_1+2\upsilon_4} \oplus
	\tau_{4\upsilon_1} 
.
\end{align}
We especially note that the last four constituents for $\tau_7$ are missing in \cite[Table 1]{Mashimo97}. 
This omission affected his calculations of the spectrum on $7$-forms (see Remark~\ref{rem8:errorMashimo}). 
Furthermore, \cite[Tables 1--28]{Mashimo06} give $\PP_{\tau}$ for each irreducible constituent $\tau$ of some $\tau_p$ such that 
\begin{equation}\label{eq8:F4unionstrings}
\widehat G_{\tau} = \bigcup_{\Lambda_0\in \PP_{\tau}} \{\pi_{k\varepsilon_1+ \Lambda_0}:k\geq0\}.
\end{equation}
Thus, the direction of all strings is $\dir := \omega_1=\varepsilon_1$. 
Then, \cite[Table 2]{Mashimo97} for $0\leq p\leq 5$ and  \cite[Tables 29--31]{Mashimo06} for $6\leq p\leq 8$, writes $\widehat G_{\tau_p}$ as a finite union of strings of representations, with the corresponding eigenvalues and multiplicity for each element in each string. 
Such information plus the dimension of the irreducible representations in the strings, fully describe the spectrum of the Hodge--Laplace operator on $p$-forms on $P^2(\mathbb O)$. 

With the  information above we can prove the validity of the strong representation-spectral converse {for several values of $p$}.

\begin{theorem}\label{thm8:F4p-iso}
The strong representation-spectral converse is valid for $(\op{F}_4, \Spin(9) ,\tau_{p})$ for {each} $p\in \{0,1,2,3,4,6,10,12,13,14,15,16\}$. 
\end{theorem}

\begin{proof}
The assertion follows by checking that, for each pair $\{\Lambda_0,\Lambda_0'\}$ in $\PP_{\tau_p}$ or $\PP_{\tau}$, the polynomials $\lambda(C,\pi_{k\dir+\Lambda_0})$ and $\lambda(C,\pi_{k\dir+\Lambda_0'})$ have only finitely many coincidences.
These polynomials are explicitly given in \cite[Table 2]{Mashimo97} for $0\leq p\leq 5$ and in \cite[Table 29--30]{Mashimo06} for $6\leq p\leq 7$. 
In this verification, each case  follows immediately from  Lemma~\ref{lem3:coincidences} and therefore is left to the reader. 
\end{proof}

The next remark shows why the cases of $p=5,7,8,9,11$ for $(\op{F}_4, \Spin(9) ,\tau_{p})$ were omitted in Theorem~\ref{thm8:F4p-iso}.

\begin{remark}\label{ex8:F4unknowncases}
It is a simple matter to check that the polynomials $\lambda(C,\pi_{k\dir+3\omega_1})= k^2+17k+66$ and $\lambda(C,\pi_{k\dir+2\omega_2})= k^2+19k+84$ have infinitely many coincidences, by using Lemma~\ref{lem3:coincidences}. 
Also, from Tables~15 and 24 in \cite{Mashimo06}, we see that $3\omega_1\in \PP_{\tau_{3\upsilon_4}}$ and $2\omega_2 \in \PP_{\tau_{\upsilon_1+\upsilon_2+\upsilon_4}}$ and consequently $3\omega_1$ and $2\omega_2$ are in $\PP_{\tau_5}$.
We conclude that the hypotheses in Theorem~\ref{thm3:tauiso=>tauequiv} are not satisfied for 
$\tau_{5}$.

Similarly, $\lambda(C,\pi_{k\dir+2\omega_1+2\omega_3})= k^2+21k+110$ and $\lambda(C,\pi_{k\dir+4\omega_3})= k^2+23k+132$ have infinitely many coincidences. 
Since $2\omega_1+2\omega_3\in \PP_{\tau_{\upsilon_1+ \upsilon_3+\upsilon_4}} \cap \PP_{\tau_{4\upsilon_1}}$ and $4\omega_3\in \PP_{\tau_{3\upsilon_1+\upsilon_4}}\cap  \PP_{\tau_{4\upsilon_1}}$, then both are in $\PP_{\tau_7}$ and $\PP_{\tau_8}$.
\end{remark}

\begin{remark}
Theorem~\ref{thm8:F4p-iso} immediately implies that the representation-spectral converse is valid for $(G,K,\tau)$ for every irreducible representation $\tau$ occurring in the decomposition of $\tau_p$ for some $p\in\{0,1,2,3,4,6\}$ (see Remark~\ref{rem2:tau-reducible}).
We point out  that also, in the cases of $p=5$ and $p=7$, the strong representation-spectral converse is still valid \emph {for any irreducible constituent of $\tau_5$ and $\tau_7$} because the strings having infinitely many coincidences in their Casimir eigenvalues (i.e.\ $\String{3\omega_1,\varepsilon_1}$ and $\String{2\omega_2, \varepsilon_1}$ for $p=5$ and  $\String{2\omega_1+2\omega_3,\varepsilon_1}$ and $\String{4\omega_3, \varepsilon_1}$ for $p=7$) are not in $\widehat G_\tau$ for a common {irreducible} $\tau\in\widehat K$ occurring in $\tau_5$ or $\tau_7$. 

The situation is different for $p=8$ since $2\omega_1+2\omega_3$ and $4\omega_3$ lie simultaneously in $\widehat G_{\tau_{4\upsilon_1}}$, hence for this constituent of $\tau_8$ we cannot make such a positive assertion. 
\end{remark}

\begin{remark}\label{rem8:errorMashimo}
Table 30 in \cite{Mashimo06} is not complete since it does not include the strings of representations corresponding to $\tau_{2\upsilon_1+\upsilon_4}$, $\tau_{\upsilon_1+\upsilon_2+\upsilon_4}$, $\tau_{\upsilon_1+\upsilon_3+\upsilon_4}$ and $\tau_{3\upsilon_1+\upsilon_4}$ (the last four irreducible constituents {of} $\tau_7$ in \eqref{eq8:F4tau_p} forgotten in \cite[Table 1]{Mashimo97}). 
Table~\ref{tableMashimo} substitutes \cite[Table 30]{Mashimo06},  using exactly the same notation as  in that article.

\begin{table}
\caption{Spectra of Laplacian $\Delta^7$ on $P^2(\mathbf{Ca})$} 
\label{tableMashimo}
$
\begin{array}{|l||c|c|c|c|c|l|}
\hline
&\multicolumn{6}{c|}{\rule{0pt}{16pt}\text{multiplicity of } V^{F_4}(\lambda+n\lambda_4)} \\ \cline{2-7}
\multicolumn{1}{|c||}{\rule{0pt}{15pt}\text{eigenvalue}}
&n=0 &n=1 &n=2 &n=3 &n\geq4 & \multicolumn{1}{c|}{\lambda} \\  \hline
\rule{0pt}{14pt}
n^2+11n   & 0&1& 2& 3&4 &										0\\ 
n^2+14n+24& 2&6&10&11&\leftarrow &								\lambda_3 \\ 
n^2+17n+54& 6&10&11&\leftarrow &\leftarrow &					2\lambda_3\\ 
n^2+20n+90& 4&5&\leftarrow &\leftarrow &\leftarrow & 			3\lambda_3 \\ 
n^2+23n+132&1&\leftarrow &\leftarrow &\leftarrow &\leftarrow & 	4\lambda_3 \\ 
n^2+15n+36& 2&7&10&\leftarrow &\leftarrow & 					\lambda_2 \\ 
n^2+18n+68& 7&10&\leftarrow &\leftarrow &\leftarrow &			\lambda_2+\lambda_3 \\  
n^2+21n+106& 3&\leftarrow &\leftarrow &\leftarrow &\leftarrow & \lambda_2+2\lambda_3 \\ 
n^2+19n+84& 2&\leftarrow &\leftarrow &\leftarrow &\leftarrow & 	2\lambda_2 \\ 
n^2+13n+18& 1&3&6&8&\leftarrow &								\lambda_1 \\ 
n^2+16n+46& 5&11&13&\leftarrow &\leftarrow & 					\lambda_1+\lambda_3 \\ 
n^2+19n+80& 6&8&\leftarrow &\leftarrow &\leftarrow & 			\lambda_1+2\lambda_3 \\
n^2+22n+120& 2&\leftarrow &\leftarrow &\leftarrow &\leftarrow &	\lambda_1+3\lambda_3 \\ 
n^2+17n+60& 4&7&\leftarrow &\leftarrow &\leftarrow & 			\lambda_1+\lambda_2 \\ 
n^2+20n+96& 3&\leftarrow &\leftarrow &\leftarrow &\leftarrow & 	\lambda_1+\lambda_2+\lambda_3 \\ 
n^2+15n+40& 1&3&4&\leftarrow &\leftarrow & 						2\lambda_1 \\ 
n^2+18n+72& 3&4&\leftarrow &\leftarrow &\leftarrow & 			2\lambda_1+\lambda_3 \\ 
n^2+21n+110&1&\leftarrow &\leftarrow &\leftarrow &\leftarrow & 	2\lambda_1+2\lambda_3 \\ 
n^2+19n+88& 1&\leftarrow &\leftarrow &\leftarrow &\leftarrow & 	2\lambda_1+\lambda_2 \\ \hline
\end{array}
$
\end{table}
\end{remark}

\bibliographystyle{plain}

\begin{thebibliography}{LMR15}
\bibitem[BH19]{ShamsHunsicker17}
	{\sc N.S. Bari, E. Hunsicker}.
	{\it Isospectrality for orbifolds lens spaces.}
	Canad. J. Math. \textbf{72}:2 (2020), 281--325.
	DOI: \href{https://doi.org/10.4153/S0008414X19000178} {10.4153/S0008414X19000178}.


\bibitem[BR11]{BhagwatRajan11}
	{\sc C. Bhagwat, C.S. Rajan}.
	{\it On a spectral analog of the strong multiplicity one theorem.}
	Int. Math. Res. Not. IMRN \textbf{2011}:18 (2011), 4059--4073.
	DOI: \href{http://dx.doi.org/10.1093/imrn/rnq243}{10.1093/imrn/rnq243}.

\bibitem[BD]{BrockerDieck}
	{\sc T. Br\"ocker, T. tom Dieck}.
	{Representations of compact Lie groups.}
	{\it Grad. Texts in Math.} \textbf{98}.
	Springer-Verlag, New York Inc, 1985.
	DOI: \href{http://dx.doi.org/10.1007/978-3-662-12918-0} {10.1007/978-3-662-12918-0}.

\bibitem[Ca05a]{Camporesi05JFA}
	{\sc R. Camporesi}.
	{\it The {H}elgason {F}ourier transform for homogeneous vector bundles over compact {R}iemannian symmetric spaces---the local theory}.
	J. Funct. Anal. \textbf{220}:1 (2005), 97--117. 
	DOI: \href{http://dx.doi.org/10.1016/j.jfa.2004.08.013} {10.1016/j.jfa.2004.08.013}.
	
\bibitem[Ca05b]{Camporesi05Pacific}
	{\sc R. Camporesi}.
	{\it A generalization of the Cartan--Helgason theorem for Riemannian symmetryc spaces of rank one.}
	Pacific J. Math. \textbf{222}:1 (2005), 1--27. 
	DOI: \href{http://dx.doi.org/10.2140/pjm.2005.222.1} {10.2140/pjm.2005.222.1}.

\bibitem[El12]{Chami12}
	{\sc F. {El Chami}}.
	{\it A branching law from {${\rm Sp}(n)$} to {${\rm Sp}(q)\times {\rm Sp}(n-q)$} and an application to {L}aplace operator spectra}.
	Indian J. Pure Appl. Math. \textbf{43}:1 (2012), 71--86. 
	DOI: \href{http://dx.doi.org/10.1007/s13226-012-0005-4} {10.1007/s13226-012-0005-4}.

\bibitem[DG89]{DeTurckGordon89}
{\sc D. DeTurck, C. Gordon}.
{\it Isospectral deformations II: Trace formulas, metrics, and potentials}.
Comm Pure Appl. Math. \textbf{42}:8 (1989), 1067--1095.
DOI: \href{http://dx.doi.org/10.1002/cpa.3160420803} {10.1002/cpa.3160420803}.

\bibitem[GW]{GoodmanWallach-bookSpringer}
	{\sc R. Goodman, N. Wallach}.
	Symmetry, Representations, and Invariants.
	{\it Grad. Texts in Math.} \textbf{255}.
	Springer-Verlag New York, 2009.
	DOI: \href{http://dx.doi.org/10.1007/978-0-387-79852-3} {10.1007/978-0-387-79852-3}.

\bibitem[GM06]{GornetMcGowan06}
	{\sc R. Gornet, J. McGowan}.
	{\it Lens Spaces, isospectral on forms but not on functions}.
	LMS J. Comput. Math. \textbf{9} (2006), 270--286.
	DOI: \href{http://dx.doi.org/10.1112/S1461157000001273} {10.1112/S1461157000001273}.

\bibitem[FH]{FultonHarris-book}
	{\sc W. Fulton, J. Harris}.
	Representation Theory, A first course.
	Springer-Verlag New York, 2004.
	DOI: \href{http://dx.doi.org/10.1007/978-1-4612-0979-9} {10.1007/978-1-4612-0979-9}.

\bibitem[Ha07]{Halima07}
	{\sc M.B. Halima}.
	{\it Branching rules for unitary groups and spectra of invariant differential operators on complex Grassmannians}.
	J. Algebra \textbf{318}:2 (2007), 520--552.
	DOI: \href{http://dx.doi.org/10.1016/j.jalgebra.2007.08.010} {10.1016/j.jalgebra.2007.08.010}.

\bibitem[HvP16]{HeckmanPruijssen16}
	{\sc G. Heckman, M. van Pruijssen}.
	{\it Matrix valued orthogonal polynomials for {G}elfand pairs of rank one}.
	Tohoku Math. J. (2) \textbf{68}:3 (2016), 407--437.
	DOI: \href{http://dx.doi.org/10.2748/tmj/1474652266} {10.2748/tmj/1474652266}.

\bibitem[Ik80]{Ikeda80_3-dimI}
	{\sc A. Ikeda}.
	{\it On the spectrum of a riemannian manifold of positive constant curvature.}
	Osaka J. Math. \textbf{17} (1980), 75--93.
	DOI: \href{http://dx.doi.org/10.18910/5171} {10.18910/5171}
	
	
\bibitem[IT78]{IkedaTaniguchi78}
	{\sc A. Ikeda, Y. Taniguchi}.
	{\it Spectra and eigenforms of the Laplacian on $S^n$ and $P^n(\mathbb C)$}.
	Osaka J. Math. \textbf{15}:3 (1978), 515--546.
	DOI: \href{http://dx.doi.org/10.18910/6956} {10.18910/6956}
	
	
\bibitem[Ke14]{Kelmer14}
	{\sc D. Kelmer}.
	{\it A refinement of strong multiplicity one for spectra of hyperbolic manifolds.}
	Trans. Amer. Math. Soc. \textbf{366} (2014), 5925--5961.
	DOI: \href{http://dx.doi.org/10.1090/S0002-9947-2014-06102-3} {10.1090/S0002-9947-2014-06102-3}.

\bibitem[Kn]{Knapp-book-beyond}
	{\sc A.W. Knapp}.
	{Lie groups beyond an introduction.}
	{\it Progr. Math.} \textbf{140}.
	Birkh\"auser Boston Inc., 2002.

\bibitem[La14]{Lauret-strongequivflat}
	{\sc E.A. Lauret}.
	{\it Representation equivalent Bieberbach groups and strongly isospectral flat manifolds}.
	Canad. Math. Bull. \textbf{57} (2014), 357--363.
	DOI: \href{http://dx.doi.org/10.4153/CMB-2013-013-2} {10.4153/CMB-2013-013-2}.
	
\bibitem[LM20]{LM-strongmultonethm}
	{\sc E.A. Lauret, R.J. Miatello}.
	{\it Strong multiplicity one theorems for locally homogeneous spaces of compact type.}
	Proc. Amer. Math. Soc. \textbf{148}:7 (2020), 3163--3173. 
	DOI: \href{http://dx.doi.org/10.1090/proc/14980} {10.1090/proc/14980}.
	
\bibitem[LMR15]{LMR-repequiv}
	{\sc E.A. Lauret, R.J. Miatello, J.P. Rossetti}.
	{\it Representation equivalence and p-spectrum of constant curvature space forms}.
	J. Geom. Anal. \textbf{25}:1 (2015), 564--591.
	DOI: \href{http://dx.doi.org/10.1007/s12220-013-9439-0} {10.1007/s12220-013-9439-0}.

\bibitem[LMR21]{LMR-SaoPaulo}
	{\sc E.A. Lauret, R.J. Miatello, J.P. Rossetti}.
	{\it Recent results on the spectra of lens spaces.}
	S\~ao Paulo J. Math. Sci. \textbf{15}:1 (2021), 240--267.
	DOI: \href{http://dx.doi.org/10.1007/s40863-019-00154-3} {10.1007/s40863-019-00154-3}.

\bibitem[LR17]{LR-fundstring}
	{\sc E.A. Lauret, F. Rossi Bertone}.
	{\it Multiplicity formulas for fundamental strings of representations of classical Lie algebras}.
	J. Math. Phys. \textbf{58} (2017), 111703. 
	DOI: \href{http://dx.doi.org/10.1063/1.4993851} {10.1063/1.4993851}.

\bibitem[Le71]{Lepowsky71}
	{\sc J. Lepowsky}.
	{\it Multiplicity formulas for certain semisimple {L}ie groups.}
	Bull. Amer. Math. Soc. \textbf{77} (1971), 601--605.
	DOI: \href{http://dx.doi.org/10.1090/S0002-9904-1971-12767-2} {10.1090/S0002-9904-1971-12767-2}.

\bibitem[Ma97]{Mashimo97}
	{\sc K. Mashimo}.
	{\it Spectra of the {L}aplacian on the {C}ayley projective plane.}
	Tsukuba J. Math. \textbf{21}:2 (1997), 367--396.
	DOI: \href{http://dx.doi.org/10.21099/tkbjm/1496163248} {10.21099/tkbjm/1496163248}.

\bibitem[Ma06]{Mashimo06}
	{\sc K. Mashimo}.
	{\it On the branching theorem of the pair {$(F_4,{\rm Spin}(9))$}.}
	Tsukuba J. Math. \textbf{30}:1 (2006), 31--47.
	DOI: \href{http://dx.doi.org/10.21099/tkbjm/1496165027} {10.21099/tkbjm/1496165027}.

\bibitem[Pe95]{Pesce95}
	{\sc H. Pesce}.
	{\it Vari\'et\'es hyperboliques et elliptiques fortement isospectrales.}
	J. Funct. Anal. \textbf{134}:2 (1995), 363--391.
	DOI: \href{http://dx.doi.org/10.1006/jfan.1995.1150} {10.1006/jfan.1995.1150}.

\bibitem[Pe96]{Pesce96}
	{\sc H. Pesce}.
	{\it Repr\'esentations relativement \'equivalentes et vari\'et\'es riemanniennes isospectrales.}
	Comment.\ Math.\ Helvetici \textbf{71} (1996), 243--268.
	DOI: \href{http://dx.doi.org/10.1007/BF02566419} {10.1007/BF02566419}.

\bibitem[SY]{SchoenYau-book}
	{\sc R. {Schoen} and S.-T. {Yau}}.
	{Lectures on differential geometry.}
	Cambridge, MA: International Press, 1994.

\bibitem[Su85]{Sunada85}
	{\sc T. Sunada}.
	{\it Riemannian coverings and isospectral manifolds.}
	Ann. of Math. (2) \textbf{121}:1 (1985), 169--186.
	DOI: \href{http://dx.doi.org/10.2307/1971195} {10.2307/1971195}.

\bibitem[Ta]{Takeuchi}
{\sc M. Takeuchi}.
Modern spherical functions. (Transl. from the Japanese by Toshinobu Nagura.)
{\it Transl. Math. Monogr.} \textbf{135}.
Amer. Math. Soc., Providence, 1994.

\bibitem[Ts81]{Tsukamoto81}
	{\sc C. Tsukamoto}.
	{\it Spectra of Laplace-Beltrami operators on $\textrm{SO}(n+2)/\textrm{SO}(2)\times\textrm{SO}(n)$ and $\textrm{Sp}(n+1)/\textrm{Sp}(1)\times\textrm{Sp}(n)$}.
	Osaka J. Math. \textbf{18}:2 (1981), 407--426.
    DOI: \href{http://dx.doi.org/10.18910/8349} {10.18910/8349}.
    
\bibitem[Wa]{Wallach-book}
	{\sc N. Wallach}.
	{Harmonic analysis on homogeneous spaces}.
	{\it Pure and Applied Mathematics} \textbf{19}.
	Marcel Dekker, Inc., New York, 1973.

\bibitem[Wo]{Wolf-book}
	{\sc J. Wolf}.
	Spaces of constant curvature. 6th ed.
	Providence, RI: AMS Chelsea Publishing, 2011.
		

\end{thebibliography}

\end{document}